\documentclass[a4paper,12pt,intlimits,oneside]{amsart}

\usepackage{enumerate}
\usepackage{amsfonts}
\usepackage{amsmath,amsthm,amssymb}

\newcommand\de{\delta}
\newcommand{\Sig}{{\bf \Sigma}}

\newcommand{\SQ}{{\bf \Sigma}_q}
\newcommand{\SQS}{{\bf \Sigma}_q^\star}

\newcommand{\comment}[1]{}

\newtheorem{theorem}{Theorem}
\newtheorem{definition}[theorem]{Definition}
\newtheorem{proposition}[theorem]{Proposition}
\newtheorem{lemma}[theorem]{Lemma}

\newtheorem{remark}[theorem]{Remark}

\newcommand\e{\varepsilon}

\newcommand\al{\alpha}

\newcommand{\NN}{\mathbb N}
\newcommand{\ZZ}{\mathbb Z}
\newcommand{\RR}{\mathbb R}
\newcommand{\CC}{\mathbb C}
\newcommand{\TT}{\mathbb T}

\newcommand{\GG}{\mathbb{G}_q}
\newcommand{\GS}{\mathbb{G}_q^{\star}}

\newcommand\PP{\mathcal{P}}

\newcounter{rev}
\newcounter{reb}
\begin{document}

\title[Integral concentration of idempotents]{Integral Concentration of idempotent trigonometric polynomials with gaps }

\author{Aline Bonami \& Szil\'{a}rd Gy. R\'{e}v\'{e}sz}

\date{\today}

\address[Aline Bonami]{
\newline\indent F\'ed\'eration Denis Poisson
\newline\indent MAPMO-UMR 6628 CNRS
\newline\indent Universit\'e d'Orl\'eans
\newline\indent 45067 Orl\'eans France.}
\email{aline.bonami@univ-orleans.fr}

\address[Szil\'ard Gy. R\'ev\'esz]{
\newline \indent A. R\'enyi Institute of Mathematics
\newline \indent Hungarian Academy of Sciences,
\newline \indent Budapest, P.O.B. 127, 1364 Hungary.}
\email{revesz@renyi.hu}

\comment{
\address{\hskip 1cm and
\newline \indent Institut Henri Poincar\'e,
\newline \indent 11 rue Pierre et Marie Curie, 75005 Paris, France}
\email{revesz@ihp.jussieu.fr}
}

\begin{abstract} We prove that for all $p>1/2$ there exists a
constant  $\gamma_p>0$ such that, for any symmetric measurable
set of positive measure $E\subset \TT$ and for any
$\gamma<\gamma_p$, there is an idempotent trigonometrical
polynomial $f$ satisfying $\int_E |f|^p
> \gamma \int_{\TT} |f|^p$. This disproves a conjecture of
Anderson, Ash, Jones, Rider and Saffari, who proved the existence
of $\gamma_p>0$ for $p>1$ and conjectured that it does not exists
for $p=1$.

Furthermore, we prove  that one can take $\gamma_p=1$ when $p>1$
is not an even integer, and that polynomials $f$ can be chosen
with arbitrarily large gaps when $p\neq 2$. This shows  striking
differences with the case $p=2$, for which the best constant is
strictly smaller than $1/2$, as it has been known for twenty
years, and for which having arbitrarily large gaps with such
concentration of the integral is not possible, according to a
classical theorem of Wiener.

We find sharper results for $0<p\leq 1$ when we restrict to open
sets, or when we enlarge the class of idempotent trigonometric
polynomials to all  positive definite ones.
\end{abstract}

\maketitle \vskip1em \noindent{\small \textbf{Mathematics Subject
Classification (2000):} Primary 42A05. Secondary 42A16, 42A61,
42A55, 42A82, 42B05.\\[1em]
}


\let\oldfootnote\thefootnote
\def\thefootnote{}

\comment{\footnotetext{Supported in part in the framework of the
Hungarian-French Scientific and Technological Governmental
Cooperation, Project \# F-10/04.}}

\footnotetext{The second author was supported in part by the
Hungarian National Foundation for Scientific Research, Project \#s
T-049301 T-049693 K-61908 and K-72731.} \footnotetext{This work
was accomplished during the second author's stay at the Institut
Henri Poincar\'e in Paris under his Marie Curie fellowship,
contract \# MEIF-CT-2005-022927.}
\let\thefootnote\oldfootnote

\newpage

\section{Introduction}\label{sec:intro}

In this work $\TT:=\RR/\ZZ$ is the circle, and $e(t):=e^{2\pi i
t}$ is the usual exponential function adjusted to interval length
$1$. We will denote $e_h$ the function $e(hx)$. For obvious
reasons of being convolution idempotents, the set
\begin{equation}\label{eq:idempotents}
\PP:=\left\{ \sum_{h\in H}e_h ~:~ H\subset \NN, ~ \# H< \infty
\right\}
\end{equation}
is called the set of \emph{(convolution-)idempotent exponential
(or trigonometric) polynomials}, or just \emph{idempotents} for
short.

Remark that we assume all frequencies of idempotents under
consideration  to be non-negative. We can do without loss of
generality since we will only be interested in the modulus of
idempotents, which is not modified by multiplication by some
exponential $e_N$. We will denote as well
\begin{equation}\label{eq:all} \mathcal{T}:=\left\{ \sum_{h\in
H}a_he_h ~:~ H\subset \NN, ~ \# H< \infty \, ; \ a_h \in
\mathbb{C}, \ h\in H \ \right\}
\end{equation}
the space of all trigonometric polynomials.
\smallskip

The starting point of our work was a conjecture in \cite{Many}
regarding the impossibility of the concentration of the integral
norm of idempotents.

Before recording the main result of the paper \cite{Many}, let us
give some notations and definitions. We first start by the notion
of concentration on symmetric open sets, for which results are
more complete, and proofs are more elementary.

A set $E$ is symmetric if  $x\in E$ implies $- x \in E$.

\begin{definition}
Let $p>0$ and $a\in\TT$. We say that there is  $p$-concentration
at $a$ if there exists a constant $c>0$ so that for any symmetric
open set $E$ that contains $a$, one can find an idempotent
$f\in\PP$ with
\begin{equation}\label{eq:Lpconcentration_a}
 \int_E |f|^p
\geq c  \int_\TT |f|^p.
\end{equation}
Moreover, the supremum of all such constants $c$ will be denoted
as $c_p(a)$: it is called the \emph{level of the $p$-concentration
at $a$}. Such an idempotent $f$ will be called a \emph {
$p$-concentrating polynomial}.
\end{definition}

\begin{definition}
Let $p>0$. We say that there is  $p$-concentration  if  there
exists a constant $c>0$ so that for any symmetric non empty open
set $E$ one can find an idempotent $f\in\PP$ with
\begin{equation}\label{eq:Lpconcentration}
\int_E |f|^p \geq c \int_\TT |f|^p.
\end{equation}
Moreover, $c_p$ will denote the supremum of all such constants
$c$. Correspondingly, $c_p$ is called the \emph{level of
$p$-concentration}. If $c_p=1$, we say that there is \emph{full
$p$-concentration}.
\end{definition}

Clearly, as remarked in \cite{DPQ}, the local constant $c_p(a)$ is
an upper semi-continuous function on $\TT$, and
$c_p=\inf_{a\in\TT}c_p(a)$.

\begin{remark}
We have taken symmetric open sets because the function $|f|$ is
even for $f\in\PP$. Without the assumption of symmetry, the
constant $c_p(a)$ would be at most $1/2$ for $a$ different from
$0$ and $1/2$. With this definition, as we will see, $c_p(a)$ and
even $c_p$ can achieve the maximal value $1$. Nevertheless, using
the alternative definition with arbitrary open sets (or just
intervals) would only mean taking half of our constants $c_p(a)$
for $a\ne 0, 1/2$ and of $c_p$.
\end{remark}

The question of $p$-concentration, and the computation or at least
estimation of the best constant $c_p$, originated from the work of
Cowling \cite{Cow}, and of Ash \cite{Ash0} on comparison of
restricted type and strong type for convolution operators. This is
described recently in the survey \cite{Ash}. It has since then
been the object of considerable interest, with improving lower
bounds obtained by Pichorides, Montgomery, Kahane and Ash, Jones
and Saffari, see \cite{first, CRMany, Many} for details. In 1983
D\'echamps-Gondim, Piquard-Lust and Queff\'elec \cite{DPQ, DPQ2}
answered a question from \cite{first}, proving the precise value
\begin{equation}\label{case2}
c_2=\sup_{0\leq x}\frac {2\sin^2 x}{\pi x}=0.46\cdots.
\end{equation}
Moreover, they obtained $c_p\geq 2^{1-\frac p2}c_2^{p/2}$ for all
$p>2$.

As in \cite{DPQ, DPQ2, CRMany, Many}, we will consider the same
notion of $p$-concentration of (convolution-)idempotents for
measurable sets, too.

\begin{definition}
Let $p>0$ and $a\in\TT$. We say that there is  $p$-concentration
\emph{for measurable sets at $a$}, if there exists a constant
$\gamma>0$ so that for any symmetric measurable set $E$, with $a$
being a density point of $E$, there exists some idempotent
$f\in\PP$ with
\begin{equation}\label{eq:Lpconcentration-meas}
\int_E |f|^p \geq \gamma \int_\TT |f|^p.
\end{equation}
The supremum of all such constants $\gamma$ will be denoted as
$\gamma_p(a)$. Furthermore, we say that there is $p$-concentration
\emph{for measurable sets} if such an inequality holds for any
symmetric measurable set $E$ of positive measure. The supremum of
all such constants is denoted by $\gamma_p$.
\end{definition}
It is clear that $p$-concentration for measurable sets implies
$p$-concentra\-tion. On the other hand it is not clear, if
$\gamma_p(a)$ is upper semicontinuous, too. If we knew this, by
our methods that would easily imply the same strength of the
results for measurable sets, as we will obtain for open sets.

The main theorem of \cite{Many} can be stated as:
\begin{theorem}[{\bf Anderson, Ash, Jones, Rider,
Saffari}]\label{th:largepconcentration} There is
$p$-concentration  for measurable sets for all $p>1$.
\end{theorem}
We also refer to them for the fact that $\gamma_2=c_2$ is given by
\eqref{case2}. The proof of \cite{ CRMany,Many} is based on the
properties of the function
\begin{equation}\label{eq:manykernel}
D_n(x) D_n(qx),
\end{equation}
where $D_n$ stands for the Dirichlet kernel. We will use the same
notation as  in \cite{Many} and define the Dirichlet kernel as
\begin{equation}\label{eq:Dndef}
D_n(x):=\sum_{\nu=0}^{n-1} e(\nu x) = e^{\pi i(n-1)x}
\frac{\sin(\pi n x)}{\sin(\pi x)}.
\end{equation}
The idea is that the first Dirichlet kernel in
\eqref{eq:manykernel} will have sufficiently peaky behavior
(regarding $|\cdot|^p$), while the second one simulates a Dirac
delta, so that the $p$-th integral outside  very close
neighborhoods of the points $k/q$ is small. They use the
multiplicative group structure of $\ZZ/q\ZZ$, when $q$ is prime,
to prove that concentration at $k/q$ and concentration at $1/q$
may be compared.

Their proof yields $p$-concentration only with $c_p\to 0$ when
$p\to1$.  Based on these and some other heuristical arguments and
calculations the authors conjectured that for $p$-concentration
the value $1$ should be a natural limit.
 We will disprove this conjecture, even for measurable sets and we will even prove more: all
concentrating idempotents can be taken with arbitrary large gaps.
Recall that the trigonometric polynomial
\begin{equation}\label{eq:gapquestion} f(x):=\sum_{k=1}^{K} a_k e(n_k
x),
\end{equation}
has gaps larger than $N$ if it satisfies the gap condition
$n_{k+1}-n_k>N$ ($k=1,\dots,K-1)$. Before describing our results
more precisely, we need other definitions.
\begin{definition} We say that there is $p$-concentration with gap
(resp. $p$- concentration with gap for measurable sets)  at $a$ if
for all $N>0$  the $p$-concentrating polynomial in
\eqref{eq:Lpconcentration_a} (resp. in
\eqref{eq:Lpconcentration-meas}) can be chosen with gap larger
than $N$. If this holds for every $a$, we say that there is
$p$-concentration with gap (resp. $p$-concentration with gap for
measurable sets). If, moreover, the constant $c$ can be taken
arbitrarily close to $1$, we say that there is full
$p$-concentration with gap (resp. $p$-concentration with gap for
measurable sets).
\end{definition}

With these definitions, we can give our main theorems.
\begin{theorem}\label{th:concentration} For all $0<p<\infty$ we
have $p$-concentration. Moreover, if $p$ is not an even integer,
then we have full concentration, i.e. $c_p=1$. When considering
even integers, we have $c_2$ given by \eqref{case2}, then
$0.495<c_4\leq 1/2$, then for all other even integers
$0.483<c_{2k}\leq 1/2$. Moreover, unless $p=2$, we have
concentration with gap at the same level of concentration. On the
other hand for $p=2$ requiring arbitrarily large gaps would
decrease the level of concentration to $0$.
\end{theorem}

For measurable sets, our results are just as good for $p>1$.
Arriving at the limits of our current methods, we leave it as an
open problem what happens for $p\leq 1/2$, and whether there is
full concentration for $1/2<p\leq 1$.
\begin{theorem}\label{th:concentration-meas} For all $1/2 < p<\infty$
we have $p$-concentration for measurable sets. If $p$ is not an
even integer, then we have  full concentration for measurable sets
when $p>1$. If $p=2$, the level of the concentration is given by
\eqref{case2}, and for $p=4$ we have $0.495<\gamma_4\leq 1/2$. For
other even integers we have uniformly $0.483<\gamma_{2k}\leq 1/2$.
Moreover, unless $p=2$, the same level of concentration can be
achieved with arbitrarily large gaps.
\end{theorem}

\comment{ We will prove that there is full $p$-concentration with
gap whenever $p<2$. For other values of $p(>2)$, we will only
prove $p$-concentration. In both cases, the central ingredient of
the proof will be the existence of peaking idempotents at $0$ or
$1/2$. So,  we consider first $p$-concentration with gap at $0$ or
$1/2$. }

 This improves considerably the constants given in
\cite{CRMany, Many}, which tend to zero when $p\to \infty$ or when
$p\to 1^+$ (however, to compare constants, be aware of the
notational difference between us and \cite{Many, CRMany}).

We postpone to the last part of the paper what concerns measurable
sets. The proofs will follow from an adaptation of the methods
that we develop for open sets, and also from the use of
diophantine approximation. As in \cite{Many}, we do not know
whether constants $\gamma_p$ and $c_p$ differ when $p\neq 2$,
except when we know that both of them are $1$, which is the case
of all $p>1$ not an even integer.

Let us hint some of the key ideas in our proofs, which may be of
independent interest. The first one is an explicit construction of
concentrating idempotents for the points $0$ and $1/2$ at a level
of concentration arbitrarily close to $1$ and with arbitrarily
large gaps. To emphasize their role in our construction, we will
term such concentrating idempotents as {\emph ``peaking
idempotents"}, or, when referring to the large gaps required, as
``gap-peaking idempotents" -- for a more precise meaning see the
beginning of \S \ref{sec:onepeak}.

\begin{proposition}\label{th:peakingatzero} For all $p>0$, except
for $p=2$, one has full $p$-concentration with gap at $0$. For
$p=2$, positive concentration with arbitrarily large gaps is
possible at neither points $a\in \TT$.
\end{proposition}

Note that, using the Dirichlet kernel that peaks at $0$, we find
full $p$-concentration at $0$ for $p>1$. For $p\leq 1$, the
Dirichlet kernel cannot be used. For a given concentration, our
examples will be obtained using idempotents of much higher degree.
So as for the behavior at point $0$ and $p>1$ different from $2$,
the novelty is the fact that the peaking polynomial may have
arbitrarily large gaps.

This is what cannot occur in $L^2$, in view of Ingham's
inequalities \cite{Ing, Z}. The somewhat surprising new fact here
is that it does occur for all other values of $p$.

Zygmund \cite[Chapter V  {\S}9, page 380]{Z} pointed out concerning
Ingham's results on essentially uniform distribution of square
integrals (norms) for Fourier series with large gaps: "Nothing
seems to be known about possible extensions to classes $L^p, ~~
p\neq 2$". To the best of our knowledge the problem has not been
addressed thus far. But now we find that an Ingham type inequality
is characteristic to the Hilbertian case, and for no $p\ne 2$ one
can have similar inequalities, not even when restricting to
idempotent polynomials.

The next proposition is even more surprising. It is the key to
full concentration at other points than $0$.
\begin{proposition}\label{prop:peakinghalf} Full $p$-concentration
with gap at $1/2$ holds whenever $p>0$ is not an even integer. On
the other hand, for $p=2k\in 2\NN$, $c_{2k}(1/2)= 1/2$.
\end{proposition}

The assertion for $p$ an even integer will follow directly from
the work of D\'echamps-Gondim, Lust-Piquard and Queff\'elec
\cite{DPQ, DPQ2}.

For $0<p<2$ we base our argument on the properties of the
bivariate idempotent $1+e(y)+e(x+2y)$.

For $p>2$, we will rely on a construction of Mockenhaupt and
Schlag, see \cite{MS}, given in their work on the Hardy-Littlewood
 majorant problem, which we describe now in its original
formulation. Following Hardy and Litlewood, $f$ is said to be a
majorant to $g$ if $|\widehat{g}|\leq \widehat{f}$. Obviously,
then $f$ is necessarily a positive definite function. The (upper)
majorization property (with constant 1) is the statement that
whenever  $f\in L^p(\TT)$ is a majorant of $g\in L^p(\TT)$, then
$\|g\|_p\leq \|f\|_p$. Hardy and Littlewood proved this for all
$p\in2\NN$. On the other hand already Hardy and Littlewood
observed that this fails for $p=3$: they took $f=1+e_1+e_3$ and
$g=1-e_1+e_3$ (where $e_k(x):=e(kx)$) and calculated that
$\|f\|_3<\|g\|_3$.

The failure of the majorization property for $p\notin 2\NN$ was
shown by Boas \cite{Boas} (see also \cite{Bac} for arbitrarily
large constants, and also \cite{Fou, Moc} for further comments and
 similar results in other groups.) Montgomery conjectured that it
fails also if we restrict to majorants belonging to $\PP$, see
\cite[p. 144]{Mon}. This has been recently proved by Mockenhaupt
and Schlag in \cite{MS}.

\begin{theorem}[{\bf Mockenhaupt \& Schlag}]\label{th:MMS} Let
$p>2$ and $p\notin 2\NN$, and let $k>p/2$ be arbitrary. Then for
the trigonometric polynomials $g:=(1+e_k)(1-e_{k+1})$ and
$f:=(1+e_k)(1+e_{k+1})$ we have $\|g\|_p >\|f\|_p$.
\end{theorem}
Our proof of Proposition \ref{prop:peakinghalf} for $p>2$ and
$p\notin 2\NN$, will be based on the construction of Mockenhaupt
and Schlag.

 Once we have our peaking polynomials at $1/2$, we
conclude in proving the following assertion.

\begin{proposition}\label{th:conditional} Let $p>0$ and assume that we
have full p-concentration with gap at $1/2$ for this value of $p$.
Then we also have $p$-concentration. Moreover, $c_p=1$ and we have
full $p$-concentration with gap.
\end{proposition}
The proof of Proposition \ref{th:conditional} consists of
considering products like
\begin{equation}\label{eq:ourkernel}
 D_{r}(s_1x)\cdots D_{r}(s_nx) T(qx),
\end{equation}
where the similarity to \eqref{eq:manykernel} may be misleading in
regard of the role of the Dirichlet kernels here: the role of the
``approximate Dirac delta" is fully placed on $T$, which is a
peaking function at  $1/2$ with large gaps that insure that the
product is still an idempotent. The first factors will be chosen
in such a way that they coincide with a power of a Dirichlet
kernel on some grid $\frac 1{2q}+\ZZ/q\ZZ$. For measurable sets,
the use of diophantine approximation forces us to take at most two
factors, resulting in the restriction $p>1/2$.

When  there is not full $p$-concentration at $1/2$, i.e. for
$p=2k$, we could not determine $c_{2k}$ precisely. Still, we
can use a peaking function at $0$, provided by Proposition
\ref{th:peakingatzero}, thus obtaining reasonable uniform bounds.

\smallskip

Our last results derive from the consideration of the class of
positive definite trigonometric polynomials
\begin{equation}\label{eq:pos-def} \PP^+:=\left\{ \sum_{h\in
H}a_he_h ~:~ H\subset \NN, ~ \# H< \infty \, ; \ a_h >0 \
{\rm{for}} \  h\in H \ \right\},
\end{equation}
for which full $p$-concentration for measurable sets can be proved
for  $p>0$ not an even integer. We then use a randomization
process to transfer this result to the class $\PP$ for $p>2$, and
then using that even to $p>1$.

\smallskip

Let us record here two remarks on further developments of the
results given in the present paper.

\begin{remark}\label{ajout} The above results are well adapted to
give counter-examples for the Wiener property,
which is concerned with the possibility of inferring $f\in
L^p(\TT)$ for positive definite functions $f$ having large gaps in
case we know $f\in L^p(I)$ on some small interval (or even
measurable set). For developments in this direction see
\cite{CRAS}. For previous counter-examples to the Wiener property
for $p>2$, see  the references cited in \cite{CRAS}, and also the
constructions given by  Erd\H os and R\'enyi \cite{ER} with an
existential proof, and, for $p>6$, by Tur\'an \cite{T} with a
concrete construction.
\end{remark}

\begin{remark}\label{L1} As seen above, the conjecture of Ash,
Anderson, Jones, Rider and Saffari on nonexistence of
$L^1$-concentration, described after Theorem
\ref{th:largepconcentration}, fails. But in a sense this is due to
a "cheating" in the extent that we can simulate powers of
Dirichlet kernels by products of their scaled versions. In a
forthcoming note \cite{L1} we show, however, that on the finite
groups $\ZZ/q\ZZ$ uniform in $q$ $L^1$ concentration does really
fail.
\end{remark}

\medskip
Let us finally fix some notations that will be used all over. We
denote
\begin{equation}\label{eq:degree-q} \mathcal{T}_q:=\left\{
\sum_{h=0}^{q-1} a_h e_h \, ; \ a_h \in\CC \ {\rm{for}} \  h=0,
\cdots, q-1 \ \right\}
\end{equation}
the space of trigonometric polynomials of degree smaller than $q$
and
\begin{equation}\label{eq:idempotent-q} \PP_q:=\left\{
\sum_{h\in H}e_h ~:~ H\subset \{0,1, \cdots q-1\}\right\}
\end{equation}
the set of idempotents of degree smaller than $q$.
\bigskip

{\bf Aknowledgement.} The authors thank Terence Tao, who suggested
the construction of peaking functions through bivariate
idempotents and Riesz products \cite{Tao}. Although Riesz products
form a well-known technique, see e.g. \cite{Bac, GR, Moc}, and
bivariate idempotents have already been occurred in the subject,
too, see \cite{CRMany, Many}, combining these for the particular
construction did not occur to us, so the present paper could not
have been written without this contribution.

The authors thank also Gerd Mockenhaupt and Wilhem Schlag for
giving them their recent manuscript on the Hardy and Littlewood
majorant problem \cite{MS}. One of their construction plays a
crucial role in this paper.


\newpage

\bigskip

\begin{center}
{\bf Part I: Limitations of full concentration.}
\end{center}

\section{Negative results regarding concentration when $p\in2\NN$}\label{sec:negative}

Let us first start with proving that in case $p=2$, requiring
arbitrarily large gaps decreases the level of concentration to
$0$, as said in Theorem \ref{th:concentration} and Proposition
\ref{th:peakingatzero} (and, consequently, in Theorem
\ref{th:concentration-meas}, too).

For this there is a well known argument. We take an interval $E$
centered at $0$ and a triangular function $\Delta$ supported by
$2E$ and equal to $1$ at zero. Let $N$  be an integer and $f$ an
idempotent with gap $N$. Then
$$
\int_E |f|^2dt\leq 2\int \Delta |f|^2dt=2\sum_m\sum_n \widehat
\Delta(m) \widehat f(n)\overline {\widehat f (n+m)}.
$$
If we write separately the term with $m=0$ and insert
$\widehat{\Delta}(0)=|E|$, then the right hand side becomes
$$
2 |E| \sum_n |\widehat{f}(n)|^2 + 2\sum_{|m|>N}\widehat \Delta(m)
\sum_n\widehat f(n)\overline {\widehat f (n+m)}.
$$
Finally, by an application of the Cauchy-Schwarz inequality,
$$
\int_E |f|^2dt\leq 2  |E| \sum_n |\widehat{f}(n)|^2 + 2
\sum_{|m|>N}|\widehat\Delta(m)| \sum_n |\widehat{f}(n)|^2.
$$
According to Parseval's identity $\int_\TT |f|^2 dt = \sum_n
|\hat{f}(n)|^2$, hence
$$
\frac{\int_E |f|^2dt}{\int_\TT |f|^2 dt} \leq 2 |E| +  2
\sum_{|m|>N} |\widehat\Delta(m)|.
$$
The last estimate can be taken arbitrarily small by taking the
interval $E$ small enough, and then the gap $N$ large enough,
using the fact that the Fourier series of $\Delta$ is absolutely
convergent. This contradicts the peaking property with gap.

\begin{remark} The same proof, using for $\Delta$ a triangular
function supported by $E$, gives the reverse inequality
$$
\frac{\int_E |f|^2dt}{\int_\TT |f|^2 dt} \geq \frac
{|E|}{2+\epsilon},
$$
valid for functions with sufficiently large gaps, depending on $E$
and $\epsilon>0$. These type of estimates are known as Ingham type
inequalities, and various generalizations have many applications
e.g. in control theory, see \cite{Kom}, \cite{TT}, \cite{TT2}. The
fact that one can have full $p$-concentration with gap at $0$ may
be interpreted as the impossibility of an Ingham type inequality
for $p\neq 2$. This settles to the negative a problem posed by
Zygmund, see Notes to Chapter V  \S 9, page 380 in \cite{Z}.
\end{remark}

\smallskip

Next, we explain how to obtain the necessary condition $c_{2k}\leq
1/2$. In fact one knows more, since this is also valid for the
problem of concentration on the  class $\PP^+$ of positive
definite exponential polynomials (see \eqref{eq:pos-def}). Let us
denote by $c_p^+$ and $c_p(a)^{+}$, as well as $\gamma_p^+$ and
$\gamma_p^+(a)$, the corresponding concentration constants, with
the class $\PP$ of idempotents replaced by the class $\PP^+$. One
has the inequalities
$$
c_p(a) \leq c_p^+(a), \quad c_p\leq c_p^+, \qquad \gamma_p(a)\leq
\gamma_p^+(a) \quad \gamma_p\leq \gamma_p^+.
$$
It was proved in \cite{DPQ, DPQ2} that $c_2^+(1/2)=1/2$. From this
we obtain that for $p=2k$ an even integer, $c_{2k}(1/2)\leq
c_{2k}^{+}(1/2) \leq 1/2$. Indeed, if $f\in\PP^+$, so is $f^k$,
and using the already known value $c_2^+(1/2)=1/2$ we infer
$c_{2k}(1/2)\leq c_2^+(1/2)=1/2$. In fact we have equality,
$$c_{2k}(1/2)=
c_{2k}^{+}(1/2)= 1/2,$$ taking the Dirichlet kernel $D_N(2x)$ as
concentrating polynomial.

While \cite{DPQ, DPQ2} gives also $c_2^+=1/2$, we do not know the
exact values of $c_{2k}$ and $c_{2k}^{+}$ for $k>1$.

We do not have any other negative result than the ones in this \S.

\bigskip

\begin{center}
{\bf Part II: Concentration on open sets.}
\end{center}

\section{Full concentration with gap and peaking functions}\label{sec:onepeak}

In this section, we will prove Proposition \ref{th:peakingatzero}
and Proposition \ref{prop:peakinghalf}. For $a=0$ or $1/2$, we are
interested in the construction of gap-peaking idempotents, that
is, for  all $\e$, $\delta$ and $N>0$, idempotent exponential
polynomials
\begin{equation}\label{eq:gapquestion2}
T(x):=\sum_{k=1}^{K} e(n_k x),
\end{equation}
 with gap condition $n_{k+1}-n_k>N$
($k=1,\dots,K)$, so that
\begin{equation}\label{eq:concentration}
\int_{a-\delta}^{a+\delta} |T|^p > (1-\e) \int_\TT |T|^p.
\end{equation}

The first step is to prove the following.

\begin{proposition}\label{prop:bividpeaking}
Let $f$ be an idempotent exponential polynomial in two variables
and of the form
\begin{equation}\label{eq:fdef}
f(x,y) = \sum_{k=1}^K e( n_k x + m_k y ),
\end{equation}
where $K\in\NN$ and $n_k, m_k\in \NN$ are two sequences of
nonnegative integers, with $m_k$  strictly increasing. Assume that
$f$ has the property that its ``marginal $p$-integral", given by
\begin{equation}\label{eq:marginalnorm}
F(x) := \int_0^1 |f(x,y)|^p dy,
\end{equation}
has a strict maximum at $a$, for $a=0$ or $a=1/2$. Then one has
full $p$-concentration with gap at the point $a$.
\end{proposition}
\begin{proof}
Choose $M$ with $0\leq m_k, n_k <M$ for all $k$ and consider the
Riesz product
\begin{equation}\label{eq:Rieszprod}
g(x) := g_{R,J}(x):= \prod_{j=1}^J f(x, R^j x)
\end{equation}
where $R$ is a very large integer, $f$ is given by \eqref{eq:fdef}
satisfying the assumption, and $J$ will be chosen later on. If we
take $R>M(J+1)$, then $g\in\PP$; moreover, $g$ will obey a gap
condition of size $N$ if $R$ is large enough depending on $J$, $M$
and $N$. Recall that the marginal $p$--integral
\eqref{eq:marginalnorm} has a strict maximum at $a$. For any fixed
interval $I$, the integral of $|g|^p$ on $I$ will approach  the
integral of $F^J$ on $I$ as $R\to\infty$. Indeed,
$$
\int_I |g|^p =  \int_I \prod_{j=1}^J |f(x, R^j x)|^p dx \,,
$$
and as the function $|f|^p \in C(\TT^2)$, we can apply Lemma
\ref{l:Rieszconv} below.
\begin{lemma}\label{l:Rieszconv}
Assume that $\varphi\in C(\TT\times\TT^{J})$. Denote the marginal
integrals by $\Phi(x):=\int_{\TT^J }\varphi(x,{\bf y})d{\bf y}$.
Then, for $E$ a measurable set of positive measure, we have
\begin{equation}\label{eq:Rieszlimit}
\lim_{n_1,n_2,\dots,n_J\to\infty} \int_E \varphi\left(x,n_1
x,n_1n_2x, \dots, n_1n_2\cdots n_Jx\right) dx = \int_E \Phi(x) dx.
\end{equation}
\end{lemma}
Here by $n_1,\dots,n_J\to\infty$ we naturally mean
$\min(n_1,\dots,n_J)\to\infty$. For the sake of remaining
self-contained, we give a proof below, even if this one is
standard, mentioned also e.g. in \cite{Moc, Mon, Bac} (for $J=1$).

\begin{proof} By density, it is sufficient to prove this for
$\varphi$ an exponential polynomial on $\TT\times\TT^{J}$. By
linearity, it is sufficient to consider a monomial. When it does
not depend on the second variable there is nothing to prove.
Assume that $\varphi(x,{\bf y})=e(kx +l_1y_1+\cdots +l_Jy_J)$,
with at least one of the $l_j$'s being nonzero. We want to prove
that
$$
\int_E \varphi \left(x,n_1 x,n_1n_2x, \dots, n_1n_2\cdots
n_Jx\right) dx \longrightarrow 0 \qquad (n_1,\dots,n_J \to
\infty).
$$
This integral is the Fourier coefficient of the characteristic
function of $E$ at the frequency $k+n_1l_1+n_1n_2l_2+\cdots +
n_1n_2\cdots n_Jl_J$, which tends to infinity for
$n_1,\dots,n_J\to\infty$. We conclude using the Riemann-Lebesgue
Lemma.
\end{proof}

\smallskip

Let us go back to our Riesz product $g$ in \eqref{eq:Rieszprod}.
Let us first choose $J$ large enough: Then $F^J$ will be
arbitrarily concentrated on $I:=[a-\delta,a+\delta]$ in integral
because $F$ has a strict global maximum at $a$. More precisely, we
fix $J$ large enough so that
$$\int_I
F^{J} >(1-\varepsilon)\int_\TT F^{J} .
$$
Once $J$ is fixed, we use Lemma \ref{l:Rieszconv} for the function
$$
\varphi(x, {\bf y}):=\prod_{j=1}^J|f(x, y_j)|^p.
$$
We know that
$$
\lim_{R\to\infty} \int_I |g_{R,J}|^p = \int_I F^{J},
$$
and the same for the integral over the whole torus. The
proposition is proved.
\end{proof}

This concludes the proof of Proposition \ref{th:peakingatzero},
assuming that the condition of Proposition \ref{prop:bividpeaking}
holds. Next we will focus on this point.

\begin{remark}The function $|f|$ is even in the sense that
$|f(-x,-y)|=|f(x,y)|$, since the quantities inside the absolute
value sign are just complex conjugates. Therefore, $F$ is even.
Moreover it can have a \emph{unique} maximum in $\TT$ if only this
maximum is either at $0$ or at $1/2$.
\end{remark}

\begin{proposition}\label{prop:peakpol2} Let
$f(x,y):=1+e(y)+e(x+2y)$. Then   the marginal integral function
$F_p(x):=\int_0^1 |f(x,y)|^pdy$ is a continuous function, which
has a unique, strict maximum at $0$ for $p>2$, while it has a
strict maximum at $1/2$ for $p<2$.
\end{proposition}
\begin{proof}
Since $F_p$ is even, it suffices to prove that it is monotonic on
$[0, \frac 12]$, with the required monotonicity. Note that
$$
|f(x,y)|=\left|2e(x/2)\cos \left(\pi(x+2 y)\right) +1\right|.
$$
So
\begin{align*}
F_p(x)&=\int_{-1/2}^{1/2} \left|2e(x/2)\cos (2\pi y) +1\right|^p
dy
\\ &= \int_{-1/4}^{1/4}\left(|2e(x/2)\cos (2\pi y) +1|^p
+|2e(x/2)\cos (2\pi y) -1|^p\right)dy.
\end{align*}
It is sufficient to show that for fixed $y\in (-\frac 14, \frac
14)$ the quantity
$$
\Phi(x,y):=|2e(x/2)\cos(2\pi y) +1|^p +|2e(x/2)\cos (2\pi y) -1|^p
$$
is monotonic in $x$ for $0<x<\frac 12$. Considering its derivative
\begin{align*}
\frac{\partial \Phi}{\partial x}(x,y)=&-2p\pi \sin(\pi x)\cos(2\pi
y)
\\ & \times \left\{ |2e(\frac x2)\cos(2\pi y) +1|^{p-2} - |2e(\frac
x2)\cos (2\pi y) -1|^{p-2} \right\}
\end{align*}
we find that its signum is the opposite of the signum of the
difference in the second line. It follows that $\Phi$, hence $F_p$
has a strict global maximum at zero when $p>2$ and a strict global
maximum at $1/2$ when $p<2$.
\end{proof}
This concludes for the existence of a peaking function at $0$ for
$p>2$, and for a peaking function at $1/2$ for $p<2$.

\smallskip

We will need the following lemma later on.
\begin{lemma}\label{smooth} The
function $F_p$ is a $\mathcal{C}^2$ function for $p>2$ and its
second derivative at $0$ is strictly negative. For all values of
$p$ it is a $\mathcal{C}^\infty$ function outside $0$. Its second
derivative at $1/2$ is strictly negative for $p<2$.
\end{lemma}
\begin{proof} For $p>2$  the smoothness of the composite function
follows from smoothness of $|\cdot|^p$. We already know from
monotonicity of $\Phi(x,y)$ for fixed $y$ that
$\Phi_{xx}^{''}(0,y)$ is non positive. Since it is clearly not
identically $0$, it is somewhere strictly negative, hence
$F^{''}_p(0)<0$. To prove that $F_p$ is a $\mathcal{C}^\infty$
function outside $0$, it is sufficient to remark that $f(x,y)$
does not vanish for $x\neq 0$. The same reasoning as above gives
the sign of the second derivative at $1/2$.
\end{proof}
\smallskip

\begin{proof}[Proof of Proposition \ref{prop:peakinghalf}]
Let us now concentrate on peaking functions at $1/2$ for $p>2$ not
an even integer and prove Proposition \ref{prop:peakinghalf}. We
will prove the following, which relies entirely on the methods of
Mockenhaupt and Schlag \cite{MS}, but tailored to our needs with
introducing also a second variable and slightly changing the
occurring idempotents, too.
\begin{proposition}\label{prop:smooth-half}
Let $p>2$ not an even integer. For $k$ an odd number that is
larger than $p/2$, the bivariate idempotent function
\begin{equation}\label{eq:MSbivariate}
g (x, y):=(1+e_1(x)e_k(y))(1+e_1(x)e_{k+1}(y))
\end{equation}
is such that its marginal integral $G_p(x):=\int_\TT |g(x,y)|^p
dy$ has a strict maximum at $1/2$. Moreover, it is a
$\mathcal{C}^4$ function, whose second derivative at $1/2$ is
strictly negative.
\end{proposition}
\begin{proof}
After a change of variables, we see that
$$
G_p(x)=4^p \int_0^1 |\cos (\pi k y)|^p\left|\cos
\left(\pi(k+1)(y-\frac x {k(k+1)})\right)\right|^pdy.
$$
The smoothness of $G_p$ follows from the fact that it is the
convolution of two functions of class $\mathcal{C}^2$. Mockenhaupt
and Schlag have computed that
$$
2^p |\cos (\pi y)|^p=\sum_{n}(-1)^n c_n e^{2 i \pi  ny}
$$
with real coefficients $c_n=c_{-n}$, such that, for non negative
$n$,
$$
c_{n+1}=\frac{n-\frac p2}{n+\frac p2 +1}\, c_n.
$$
In the convolution, only frequencies that are multiples of both
$k$ and $k+1$ are present, so that
$$
G_p(x)=\sum_{n}(-1)^n c_{kn}c_{(k+1)n}e^{2 i \pi nx}.
$$
Indeed, the Fourier coefficient $\widehat {G_p}(n)$ is equal to
$c_m c_{m'}$, where $km=(k+1)m'$, and $n=m'/k$, which gives also
$m=(k+1)n$.

Now, looking at the inductive formula for the coefficients, and
using the fact that all $c_{kn}c_{(k+1)n}$ are positive for
$k>p/2$, we find that $G_p$ is maximum when $e^{2 i \pi
nx}=(-1)^n$ for all $n$, that is, for $x=1/2$. The computation of
the Fourier series of its second derivative implies that it is
strictly negative at this point.
\end{proof}

It remains to prove that we have the gap peaking property at $0$
for $0<p<2$. It could be deduced from the theorems below, but we
can also build on the construction of Mockhenhaupt and Schlag.
Indeed, consider for $0<p<2$, the bivariate idempotent
$$
h(x, y):=(1+e_1(y))(1+e_1(x)e_{3}(y)).
$$
Using the computations of Mockenhaupt and Schlag, similarly to the
above it is again straightforward to see that the $p$-th marginal
integral $H_p(x):=\int_{\TT} |h(x,y)|^pdy$ has a strict maximum at
$0$.

This concludes the proof of Proposition \ref{prop:peakinghalf}.
\end{proof}

\section{Restriction to a discrete problem of concentration}\label{sec:oneconc}

The second step of our proof consists of restricting the problem
of $p$-concentration of an idempotent polynomial on a small
interval into the one of concentration of an idempotent polynomial
at one point of either of the two discrete grids
\begin{equation}\label{def:grid} \mathbb{G}_q:=
\frac 1q \ZZ/q\ZZ \qquad \qquad \mathbb{G}_q^\star:=\frac 1{2q} +
\frac 1q \ZZ/q\ZZ.
\end{equation}
The idea is that if we take a gap-peaking polynomial $T$, then
multiplication by $T(qx)$ will concentrate integrals on a
neighborhood of the grid: for the first grid we need $T$ to be
peaking at $0$, and for the second one we need $T$ to do so at
$1/2$.

\begin{definition}\label{projection} For $f\in\mathcal{T}$ we
denote by $\mathbf{\Pi}_q(f)$ the polynomial in $\mathcal{T}_q$
which coincides with $f$ on the grid $\mathbb{G}_q$, that is, the
polynomial having Fourier coefficients
$$\widehat{\mathbf{\Pi}_q(f)}(k):=\sum_{j\in \NN} \widehat f(k+jq), \qquad k=0,1,\cdots,
q-1.$$
\end{definition}
In particular, if $f$ is positive definite, so is
$\mathbf{\Pi}_q(f)$. However, in general the class of idempotent
polynomials is not preserved by this projection.

\medskip

Let us first define concentration  on $\mathbb{G}_q$.

\begin{definition}
We shall say that there is $p$-concentration at $a/q$ on
$\mathbb{G}_q$ with constant $c>0$ if there exists an idempotent
polynomial $R$ such that
\begin{equation}\label{eq:discreteconc}
\left|R\left(\frac aq\right) \right|^p > c \sum_{k=0}^{q-1}
\left|R\left(\frac kq\right) \right|^p.
\end{equation}
\end{definition}
The next well-known lemma (see \cite{DPQ, Many} etc.) allows to
restrict to $a=1$.
\begin{lemma}\label{l:homominvariance}
Assume that there is $p$-concentration at $1/q$ on $\mathbb{G}_q$
with constant $c$, that is, with some appropriate idempotent $R$
we have
\begin{equation}\label{eq:discreteconc1}
\left|R\left(\frac 1q\right) \right|^p > c \sum_{k=0}^{q-1}
\left|R\left(\frac kq\right) \right|^p.
\end{equation}
Let now $a\in \NN$, $0<a<q$ be a natural number so that $a$ and
$q$ are relatively prime. Then there is also $p$-concentration at
$a/q$ on $\mathbb{G}_q$ with constant $c$: that is,
\eqref{eq:discreteconc1} implies \eqref{eq:discreteconc} with some
appropriately chosen (possibly different) idempotent $R$.
\end{lemma}

\begin{proof}
Let $Q$ be the idempotent that satisfies (\ref{eq:discreteconc1}).
 Let now
$a\not\equiv 0,1$ (mod $q$, of course) be another value, coprime
to $q$. We then have a multiplicative inverse $b$ of $a$ $\mod q$
so that $1\leq b < q$ and $ab \equiv 1$ mod $q$. With this
particular $b$ we can consider
\begin{equation}\label{eq:Rdef}
R(x):=Q(bx).
\end{equation}
Clearly we have $R(0)=Q(0)$, $R(a/q)=Q(ab/q)=Q(1/q)$, and the
values of $R(j/q)=Q(jb/q)$ with $j=0,\dots,q-1$ will cover all
values of $Q(k/q)$ with $k=0,1,\dots,q-1$, exactly once each.
Therefore, we conclude that (\ref{eq:discreteconc}) holds with $a$
and $R$.
\end{proof}
\begin{remark}\label{degree} If $Q$ is in $\PP_q$, then instead of $Q(bx)$ we can take
for $R$ the polynomial in $\mathcal{T}_q$ which coincides with
$Q(bx)$ on the grid $\mathbb{G}_q$, that is, the polynomial
$\mathbf{\Pi}_q(Q(b\,\cdot))$ of Definition \ref{projection}.
Indeed, it is also an idempotent polynomial since $b$ and $q$ are
coprime.
\end{remark}

So now it makes sense to formally define the following
concentration coefficient.
\begin{definition}\label{def:cpnumq} We define, for $q\in\NN$,
\begin{equation}\label{eq:cpq}
c_p^\sharp(q) := \sup_{R\in \PP} \frac{\left|R\left(\frac
1q\right) \right|^p}{ \sum_{k=0}^{q-1} \left|R\left(\frac
kq\right) \right|^p},
\end{equation}
and
\begin{equation}\label{eq:cplim}
c_p^\sharp := \liminf_{q\rightarrow \infty} c_p^\sharp(q).
\end{equation}
\end{definition}

We want to extend concentration results on discrete point grids to
the whole of $\TT$, and keep track of constants. We state this as
a proposition.
\begin{proposition}\label{discret2cont} Let $p>0$ be such that
there is full $p$-concentration with gap at $0$. If $c_p^\sharp
>0$, then $p$-concentration holds for the whole of $\TT$, and we
have the inequality
\begin{equation}\label{constant}
c_p \geq 2  c_p^\sharp .
\end{equation}
Moreover, the same level of concentration holds with gap.
\end{proposition}
\begin{proof}
Let us fix a symmetric open set $E$ and construct a related
peaking idempotent. First, there exists some interval $J:= \left[
\frac aq - \frac{1}{2q},\frac aq + \frac{1}{2q} \right]$ with
$(a,q)=1$, such that $J$ and $-J$ are contained in $E$. We fix $R$
that gives the $p$-concentration at $a/q$ on $\mathbb{G}_q$ with a
constant $C$: this can be done with $C$ arbitrarily close to
$c_p^\sharp(q)$ in view of Lemma \ref{l:homominvariance}.

Now, let $\varepsilon$ be given. By uniform continuity we may
choose $0<\delta<1/2$   so that we have the inequalities
\begin{equation}\label{main-rest}
  |R(t+a/q)|^p \geq |R(a/q)|^p
-\varepsilon\left|R(a/q)\right|^p \;,  \qquad (|t|\leq \delta/q)
\end{equation}
and, for $k=0,1,\cdots, q-1$,
$$|R(t+k/q)|^p \leq
|R(k/q)|^p+ \e |R(0)|^p , \qquad (|t|\leq \delta/q)$$
 which implies immediately
\begin{equation}\label{eq:RmaxonIk}
\sum_{k=0}^{q-1}|R(t+k/q)|^p \leq
(1+q\e)\sum_{k=0}^{q-1}|R(k/q)|^p. \qquad (|t|\leq \delta/q)
\end{equation}
Once $\delta$ is chosen, we will take $T$ a gap-peaking idempotent
at $0$, provided by Proposition \ref{th:peakingatzero} -- compare
also \eqref{eq:gapquestion2}-\eqref{eq:concentration} -- with the
given $\varepsilon$, $\delta$ as above, and $N$ larger than the
degree of $R$, so that
\begin{equation}\label{eq:Sdef}
S(x):=R(x)T(qx)
\end{equation}
is an idempotent, too.  It remains to show
\begin{equation}\label{eq:SonTT}
2C\int_{\TT} |S|^p\leq \kappa(\e)\int_E |S|^p ,
\end{equation}
with $\kappa(\varepsilon)$ getting arbitrarily close to $1$ when
$\varepsilon$ is chosen appropriately small.

Denoting $\tau^p:=\int_{\TT}|T|^p$ and $I:= [ \frac aq -
\frac{\delta}{2q},\frac aq + \frac{\delta}{2q}]$, we find
\begin{align}\label{eq:SonJ}
\frac 12 \int_E |S|^p \geq \int_{J} |S|^p & \geq (1-
\varepsilon)\left|R(a/q)\right|^p \int_{I} |T(qx)|^pdx \notag \\ &
\geq (1- \varepsilon)\left|R(a/q)\right|^p ~~\frac{1}{q}
\int_{-\delta}^{\delta} |T|^p \notag \\ & \geq
\frac{(1-\varepsilon)^2\tau^p}{q} \left|R(a/q)\right|^p.
\end{align}
We now  estimate the whole integral of $|S|^p$. We define the
intervals
$$
J_k:= \left[ \frac kq - \frac{1}{2q},\frac kq + \frac{1}{2q}
\right], \quad I_k:=\left[ \frac kq - \frac{\delta}{q},\frac kq +
\frac{\delta}{q} \right] \quad (k=0,\dots,q-1).
$$ Then, if
we proceed as in (\ref{eq:SonJ}), using (\ref{eq:RmaxonIk}) this
time, we find that $$\sum_{k=0}^{q-1}\int_{I_k} |S|^p=\int_{I_0}
\sum_{k=0}^{q-1}|R(t+k/q)|^p |T(qt)|^p\leq
\frac{\tau^p}{q}(1+q\e)\sum_{k=0}^{q-1}|R(k/q)|^p,$$ while
\begin{align*}
\int_{J_k\setminus I_k}|S|^p&\leq 2\left|R(0)\right|^p \int_{\frac
kq+\frac \delta q}^{\frac kq+\frac 1{2 q}} |T(qx)|^p dx=\frac
2q\left|R(0)\right|^p \int_{\frac \delta q}^{\frac 1{2}} |T(x)|^p
dx\\&\leq \frac {\varepsilon \tau^p}q \left|R(0)\right|^p\leq
\frac{\varepsilon \tau^p}{q} \sum_{k=0}^{q-1}|R(k/q)|^p.
\end{align*}Taking the sum over $k$ for the last integrals
and adding the above sum for integrals over the $I_k$'s,  we
obtain the estimate
\begin{equation}\label{eq:SonI}
\int_{\TT} |S|^p \leq  \frac{\tau^p}{q}(1+2q\varepsilon)\sum_{k}
|R(k/q)|^p.
\end{equation}
Combining \eqref{eq:SonJ} and \eqref{eq:SonI}, \eqref{eq:SonTT}
obtains with
$\kappa(\varepsilon):=(1-\varepsilon)^{-2}(1+2q\varepsilon)$.
\smallskip

Let us finally prove $p$-concentration with gap. It is sufficient
to remark that instead of taking the polynomial $R$ in
(\ref{eq:Sdef}) we could have as well taken the polynomial
$R((Mq+1)x)$, with $M$ arbitrarily large. From this point, the
proof is identical, since the two polynomials take the same values
on the grid. If the gaps of the peaking idempotent $T$ are taken
large enough, then $S$ will have gaps larger than $M$.
\end{proof}

We can modify slightly the previous proof of Proposition
\ref{discret2cont} to prove concentration results on the
corresponding second grid, using the peaking property with gap at
$1/2$ instead of $0$.
\begin{definition}
We shall say that there is $p$-concentration at $\frac{2a+1}{2q}$
on the grid $\mathbb{G}_q^\star$ with constant $c$ if there exists
an idempotent polynomial $R$ such that
\begin{equation}\label{eq:qqconc}
\left|R\left(\frac {2a+1}{2q}\right) \right|^p > c
\sum_{k=0}^{q-1} \left|R\left(\frac {2k+1}{2q}\right) \right|^p.
\end{equation}
\end{definition}
Remark that in particular we restrict to idempotents $R$ that do
not vanish identically on the grid under consideration, which we
assume in the following definition.
\begin{definition}\label{def:cpetq} If $q\in\NN$, then we
define
\begin{equation}\label{eq:cpqq}
c_p^{\star}(q) := \sup_{R\in \PP} \frac{\left|R\left(\frac
1{2q}\right) \right|^p}{ \sum_{k=0}^{q-1} \left|R\left(\frac
{2k+1}{2q}\right) \right|^p}
\end{equation}
and
\begin{equation}\label{eq:cplim2}
c_p^{\star} := \liminf_{q\rightarrow \infty} c_p^{\star}(q).
\end{equation}
\end{definition}

Again, the first step is to restrict to $1/(2q)$.
\begin{lemma}\label{l:homominvqq}
Assume that there is $p$-concentration at $1/(2q)$ on
$\mathbb{G}_q^\star$ with constant $c$. Let now $a\in \NN$, $0\leq
a<q$ be so that $2a+1$ and $q$ are relatively prime. Then there is
also $p$-concentration at $(2a+1)/(2q)$ on the grid
$\mathbb{G}_q^\star$ with the same constant $c$.
\end{lemma}
\begin{proof}
Let $Q$ be the idempotent that satisfies \eqref{eq:qqconc} with
$a=0$. We then have a multiplicative inverse $b$ of $2a+1$ $\mod
2q$ so that $1\leq b < 2q$ and $(2a+1)b \equiv 1$ mod $2q$; hence,
in particular, also $b$ is odd. Now with this particular $b$ we
can consider $R(x):=Q(bx)$ exactly as before in \eqref{eq:Rdef}.

Clearly we have $R(0)=Q(0)$,
$R((2a+1)/(2q))=Q((2a+1)b/(2q))=Q(1/(2q))$, and the values of
$R(j/(2q))=Q(jb/(2q))$ with $j=0,\dots,2q-1$ will cover all values
of $Q(k/(2q))$ with $k=0,1,\dots,2q-1$, exactly once each, and
such a way, that odd $j$'s correspond to odd $k$'s. Therefore, we
conclude that (\ref{eq:qqconc}) holds with $2a+1$ and $R$.
\end{proof}
\begin{remark}\label{degree-half} As in Remark \ref{degree}, if
$Q$ is in $\PP_{2q}$ then instead of $Q(bx)$ we can take for $R$
the polynomial $\mathbf{\Pi}_{2q}( Q( b\, \cdot))$, which
coincides with $Q(bx)$ at each point of the grid
$\mathbb{G}_{2q}$, hence a priori on $\GS$.
\end{remark}

The corresponding proposition goes as follows:
\begin{proposition}\label{half}
Let $p>0$ be such that there is full $p$-concentration with gap at
$1/2$. If $c_p^{\star}>0$, then $p$-concentration holds for the
whole of $\TT$ and we have the inequality
\begin{equation}\label{constant_second}
c_p \geq 2  c_p^{\star}.
\end{equation}
Moreover, the same property holds with arbitrarily large gaps.
\end{proposition}

\begin{proof}
Similarly to the above, it suffices to derive the concentration
phenomenon for the symmetrized of an  interval $J:=[\frac aq ,
\frac {(a+1)}{q} ]$ for $q$ a sufficiently large number, $2a+1$
coprime to $2q$.

In this setup for any $c<c_p^{\star}(q)$ Lemma \ref{l:homominvqq}
leads to the inequality
\begin{equation}\label{conc_half}
\left|R\left(\frac {2a+1}{2q}\right) \right|^p > c
\sum_{k=0}^{q-1} \left|R\left(\frac {2k+1}{2q}\right) \right|^p.
\end{equation}
with an appropriate $R\in\PP$.

At this point, the proof is exactly the same as the one of the
previous proposition, considering intervals $I_k$ centered at
$(2k+1)/(2q)$ with radius $\delta/q$, with $\delta$ small enough
so that $R$ is nearly constant on $I_k$, and then considering
$S(x):=R(x)\cdot T(qx)$ again, where $T$ is now a gap-peaking
idempotent at $1/2$, with gaps sufficiently large, so that $S$ is
still an idempotent. Using the fact that outside $I_k$ but within
$(k/q, (k+1)/q)$, the integral of $T$ is arbitrarily small in view
of the peaking property at $1/2$, we obtain the assertion as
before. The only difference is the fact that $0$ is no more in the
grid, so that the quotient of $R(0)^p$ with $\sum_{k=0}^{q-1}
\left|R\left( (2k+1)/(2q)\right) \right|^p$ appears in the rests,
but does not change the limit since it remains fixed while $\e$
tends to $0$.

The $p$-concentration with gap at the same level of concentration
is obtained also in a similar way.
\end{proof}

\section{ $p$-concentration by means of peaking at $1/2$}\label{sec:pcondiconc}

We now prove the  part of Theorem \ref{th:concentration}
concerning $p$ not an even integer, which we state separately for
the reader's convenience. The  following proof contains also the
one of Proposition \ref{th:conditional}, which we gave in the
introduction as a hint for the methods.

\begin{proposition}\label{prop:pcondiconc} Let $p>0$ be a given value
for which there is full $p$-concentration with gap at $1/2$. Then
for each nonempty symmetric open set $E\subset \TT$ and each
constant $c<1$ we can find an idempotent $S\in\PP$ with the
property that
\begin{equation}\label{eq:pnormconcentration}
 \int_E |S|^p  > c\int_\TT |S|^p.
\end{equation}
 Moreover, $S$ may be chosen with arbitrarily large gaps.
\end{proposition}

\begin{proof} By Proposition \ref{half},
it is sufficient to prove that $c_p^{\star}=1/2$, that is,

\begin{equation}\label{eq:limsup}
\liminf_{q\rightarrow \infty}\sup_{R\in\PP} \frac
{\left|R\left(\frac 1{2q}\right) \right|^p }{ \sum_{k=0}^{q-1}
\left|R\left(\frac {2k+1}{2q}\right) \right|^p}=\frac 12.
\end{equation}
We will restrict to a sub-family of polynomials in $\PP$, obtained
by products of Dirichlet kernels. Observe first that for $r<q$,
the product
$$
D_r(x)\prod_{l=1}^{L-1} D_r\left(((2q)^l+1)x\right)
$$
is also an idempotent polynomial, the modulus of which coincides
with the $L$-th power of $|D_r|$ on the grid under consideration.
So we are to prove also the last inequality in
\begin{equation}\label{eq:liminf}
\frac{1}{c_p^\star} = \limsup_{q\to\infty} \frac{1}{2c_p^\star(q)}
\leq \inf_L\limsup_{q\rightarrow \infty}\min_{r<q} ~~ \frac12
\cdot \frac { \sum_{k=0}^{q-1} \left|D_r\left(\frac
{2k+1}{2q}\right) \right|^{Lp}}{\left|D_r\left(\frac 1{2q}\right)
\right|^{Lp} }\leq 1.
\end{equation}
Let us define
\begin{equation}\label{A-half}
A(\lambda, r,q):=\left|\frac{\sin\left(\frac{\pi}{2q}\right)}
{\sin\left(\frac{r\pi}{2q}\right)}\right|^{\lambda} \sum_{k\in
\NN\,; k<q/2}\left|\frac{\sin\left(\frac{(2k+1)r\pi}{2q}\right)}
{\sin\left(\frac{(2k+1)\pi}{2q}\right)}\right|^{\lambda}.
\end{equation}
Substituting the explicit value of $D_r$ and using parity, the
quantity appearing in the left hand side of (\ref{eq:liminf}) can
be written as $A(Lp,r,q)$ for $q$ even. When $q$ is odd, we have
to subtract half of the term obtained for $k=(q-1)/2$, which gives
only a $0$ contribution to the limit below. In any case, we have
the inequality
\begin{equation}\label{eq:maj-A-half}
\left( \frac{1}{2c_p^\star(q)} \leq \right) \frac12  \frac {
\sum_{k=0}^{q-1} \left|D_r\left(\frac {2k+1}{2q}\right)
\right|^{Lp}}{\left|D_r\left(\frac 1{2q}\right) \right|^{Lp} }\leq
A(Lp, r, q).
\end{equation}

We then have the following lemma.
\begin{lemma}\label{l:majA}
For fixed $\lambda>1$, we have the inequality
\begin{equation}\label{eq:majAlimit}
\limsup_{q\rightarrow \infty}\min_{r<q}A( \lambda, r, q)\leq
\inf_{0<t<1/2} A(\lambda, t),
\end{equation}
where
\begin{equation}\label{eq:Alimit}
A(\lambda, t):=\frac 1 {\left(\sin(\pi t) \right)^{\lambda}}
\sum_{k=0}^{\infty}\left|\frac{\sin\left((2k+1)\pi t\right)}
{2k+1}\right|^{\lambda}.
\end{equation}
\end{lemma}
\begin{proof}
 Let us  fix  $t\in (0,1/2)$, and consider the
limit of $A(\lambda, 2[qt], q)$ when $q$ tends to $\infty$. It has
the same limit as
$$
\left|\frac{\frac{\pi}{2}} {\sin\left(\pi
t\right)}\right|^{\lambda}
\sum_{k=0}^{(q-1)/2}\left|\frac{\sin\left(\frac{(2k+1)[qt]\pi}{q}\right)}
{q\sin\left(\frac{(2k+1)\pi}{2q}\right)}\right|^{\lambda}.
$$
As $q\sin(\frac{(2k+1)\pi}{2q}) \geq (2k+1) $, Lebesgue's theorem
for series justifies taking the limit termwise. This concludes the
proof of the lemma.
\end{proof}

So in view of Lemma \ref{l:majA} $1/(2 c_p^{\star}(q)) \leq
\inf_L\inf_t A(Lp,t)$. If we take $t=1/4$, all the absolute values
of the occurring sines in $A(\lambda,t)$ are equal, hence cancel
out. It remains
$$
A(\lambda, 1/4)=\sum_k (2k+1)^{-\lambda}=(1-2^{-\lambda})\zeta(\lambda).
$$
Now we can take $L$, or $\lambda=Lp$, arbitrarily large.
Therefore, the infimum in \eqref{eq:liminf} is just $1$.
\end{proof}
\medskip

Note that we found that $1/c_p^\star \leq \inf_L \inf_t A(Lp,t)$
holds always.

Let us conclude this section by a remark that will be used later
on for measurable sets, where we will not be able to consider
large products of Dirichlet kernels for $p\leq 1$, and will have
to restrict to  two factors, that is, take $L=2$. Observe that
each term  $\left|\frac{\sin\left((2k+1)\pi t \right)} {(2k+1)
\sin (\pi t)}\right|$ is below $1$, so that $A(\lambda,t)$ and
$\inf_t A(\lambda,t)$ are strictly decreasing functions of
$\lambda$.

Moreover, $\inf_{0<t<1/2} A(2,t)$ can be computed explicitly. To
compute the summation, we can use Plancherel Formula once we have
recognized the Fourier coefficients (at $k$ and $-k$) of the
function
$$\frac {\pi}2 \left
(\chi_{[-t/2,t/2]}(x)-\chi_{[-t/2,t/2]}(x-1/2)\right).$$ It
follows that
\begin{equation}\label{exact}
A(2,t)=\frac {\pi^2 t}{4\sin^2(\pi t)}.
\end{equation}
Substituting $x=\pi t$ and recalling \eqref{case2} we find
$1/\min_t A(2,t)=2c_2\approx 0.92...$, which is already much
larger than $1/2$, and close to $1$.

\section{Uniform lower bounds for $p$-concentration }\label{sec:unilower}

We now prove the lower estimation in the $p\in 2\NN$ part of
Theorem \ref{th:concentration}. We proceed as in the last section,
using Proposition \ref{discret2cont} instead of Proposition
\ref{half}, since we have now gap-peaking idempotents at $0$.
Similarly to the above, we consider a product of Dirichlet
kernels:
\begin{equation}
\label{prod-idem} R(x):=D_r(x)\prod_{\ell=1}^{L-1}
D_r((q^\ell+1)x).
\end{equation}
We have to consider the quantities \eqref{eq:cpq} and
\eqref{eq:cplim}, i.e. we are to calculate
\begin{equation}\label{liminf}
\frac2{c_p}\leq  \limsup_{q\rightarrow \infty}
\frac1{c_p^{\sharp}(q)} \leq \inf_L\limsup_{q\rightarrow
\infty}\min_{r<q} \frac { \sum_{k=0}^{q-1} \left|D_r\left(\frac
{k}{q}\right) \right|^{Lp}}{\left|D_r\left(\frac 1{q}\right)
\right|^{Lp} }.
\end{equation}
\comment{Remark that we can assume that $r\leq q/2$. Indeed,
$D_r+e_rD_{q-r}=D_q$, which vanishes on the grid except at $0$. So
$|D_r|=|D_{q-r}|$ on the grid except for $0$, and among the two
indices $r$ and $q-r$ it is the smaller one that gives the smaller
quotient.}

As before, in order to estimate the quotient in \eqref{liminf} we
have to consider the equivalent quantity $B(Lp, r,q)$ defined by
\begin{equation}\label{eq:Batzero}
B(\lambda,r,q):= \left|\frac{r\sin\left(\frac{\pi}{q}\right)}
{\sin\left(\frac{r\pi}{q}\right)}\right|^{\lambda} +
2\left|\frac{\sin\left(\frac{\pi}{q}\right)}
{\sin\left(\frac{r\pi}{q}\right)}\right|^{\lambda}
\sum_{k=1}^{q/2}\left|\frac{\sin\left(\frac{kr\pi}{q}\right)}
{\sin\left(\frac{k\pi}{q}\right)}\right|^{\lambda}.
\end{equation}
 We then have the following lemma.
\begin{lemma}\label{l:majB}
For fixed $\lambda>1$, we have the inequality
\begin{equation}\label{eq:majBlimit}
\limsup_{q\rightarrow \infty}\min_{r<q}B( \lambda, r, q)\leq
\inf_{0<t<1/2}B(\lambda, t),
\end{equation} where
\begin{equation}\label{eq:Bdef}
B(\lambda,t):=\left(\frac{\pi t}{\sin \pi
t}\right)^{\lambda}\left(1+2\sum_{k=1}^{\infty}
\left|\frac{\sin\left(k\pi t\right)}{k\pi
t}\right|^{\lambda}\right).
\end{equation}
\end{lemma}
\begin{proof}
For  fixed $t\in (0,1/2)$, the left hand side of
\eqref{eq:majBlimit} is bounded by the value that we obtain when
letting $q\to\infty$ with $r/q$ tending to $t$ at the same time.
We conclude as in Lemma \ref{l:majA}.
\end{proof}

Let us define for any fixed value of $\kappa>0$, the quantity
\begin{equation}\label{eq:BpAlambda}
\beta(\kappa):=\limsup_{\lambda\mapsto \infty}B\left (\lambda,
\kappa\sqrt{6 /\lambda }\right),
\end{equation}
which will be useful later on,  since $2/c_p\leq \inf_L \inf_t
B(Lp,t) \leq \beta(\kappa)$. For fixed $s$, the quantity
$\big(\sqrt{\lambda}/s ~\cdot ~ \sin (s/\sqrt
\lambda)\big)^\lambda$ tends to $\exp(-s^2/6)$. We use this for
the computation of $\beta (\kappa)$ and see that the first factor
of \eqref{eq:Bdef} tends to $\exp(\kappa^2\pi^2)$.

Applying the well-known Weierstrass product for $\sin$ we get
$$
\log  \left(\frac{\sin x}{x}\right) =\sum_{n=1}^{\infty}
\log\left( 1-\frac{x^2}{n^2\pi^2}\right)\leq  -
\left(\sum_{n=1}^{\infty}\frac 1{n^2}\right)\frac
{x^{2}}{\pi^{2}}=- \frac{x^2}{6}.
$$
For the $\log$ function here we must restrict to $0<x<\pi$: that
provides us the useful inequality
\begin{equation}\label{eq:sinexp}\notag
\frac{\sin x}{x} \leq \exp\left(-\frac{x^2}{6}\right)\qquad
(0<x<\pi),
\end{equation}
what we apply in the second factor of \eqref{eq:Bdef} for the
range $1\leq k < 1/t$. Thus (at the end extending the sum up to
$\infty$) we are led to
\begin{equation}\label{eq:Dsmallk}
\sum_{k<1/t} \left|\frac{\sin\left(k\pi t\right)}{k\pi
t}\right|^{\lambda} \leq \sum_{k<1/t} \exp\left( -\lambda
\frac{k^2\pi^2 t^2}{6} \right)
\leq \sum_{k=1}^\infty e^{-\kappa^2 k^2\pi^2}.
\end{equation}
Using the trivial bound $|\sin u|\leq 1$, the tail sum can be
estimated as
\begin{equation}\label{eq:Dklarge}
\sum_{k\geq 1/t} \left|\frac{\sin\left(k\pi t\right)}{k\pi
t}\right|^{\lambda} \leq (\pi t)^{-\lambda} \left( t^{\lambda} +
\int_{1/t}^{\infty} \frac{du}{u^{\lambda}}\right)=
\pi^{-\lambda}\left(1+\frac{1/t}{\lambda-1}\right),
\end{equation}
which tends to $0$ with $t=\kappa\sqrt{6/\lambda}$ and
$\lambda\to\infty$.

Collecting the above estimates for
$\beta:=\inf_{\kappa>0}\beta(\kappa)$, we are led to
\begin{equation}\label{majB}
\beta \leq \inf_{\kappa>0} e^{\pi^2\kappa^2}
\left\{1+2\sum_{k=1}^{\infty} e^{-\kappa^2 k^2 \pi^2 }\right\} .
\end{equation}
Note that the sum in the last curly brackets is well-known as
Jacobi's theta function. Choosing here $\kappa=0.225$, we can
compute $\beta\leq 4.13273$, which leads to $c_p\geq 2/\beta \geq
0.48394$, surprisingly close to the theoretical upper bound of
$1/2$.

\medskip

The computation  of $\inf_{0<t<1/2}B(\lambda, t)$ can be executed
explicitly for $\lambda=4$. We recognize the Fourier coefficients
of the convolution product $\chi_{[-t/2,t/2]}*\chi_{[-t/2,t/2]}$,
whose $L^2$ norm is equal to $(2t^3/3)^{1/2}$. Then we use the
Plancherel Formula and obtain that
\begin{equation}\label{p=4}
c_4\geq \max_{0<t<1/2} \frac{3 \left(\sin ^4 (\pi t)
\right)}{\pi^4 t^3 } > 0.495,
\end{equation}
the concrete numerical value having been obtained for the choice
of $t = 0.267$.

\medskip

Comparing the results of the last two sections, it should become
clear why gap-peaking at $1/2$ is even more useful for us, than
gap-peaking at $0$. Indeed, once we can apply gap-peaking at 1/2,
we are able to consider $\GS$ in place of $\GG$: and that means
that instead of the \emph{second largest term} $|D_r(1/q)|$, we
can consider \emph{the very largest term} $|D_r(1/2q)|$ in
comparison to the whole grid sum. Thus in the translated grid case
we can take advantage of considering arbitrarily large powers $L$,
eventually killing all other terms compared to our
$|D_r(1/2q)|^L$, while in the original grid $\GG$ this is subject
to a fine balance, restricted by the necessity of keeping control
of the dominance of the very largest term $D_r(0)^L$.

\bigskip

\begin{centerline} {\bf Part III : Concentration for measurable
sets}
\end{centerline}

\bigskip

We will go back to all steps of the previous proofs in order to
partly generalize the results to measurable sets. We start by
using the theorem of Khintchine on diophantine approximation, see
\cite{Khi}. We prove that a symmetric measurable set of positive
measure contains large parts of intervals which are centered at a
point of one of the two grids, $\GG$ or $\GS$. This is done in
Section \ref{sec:Diophantine}. Then in Section \ref{s:peak-meas}
we prove the gap-peaking property at $0$ or $1/2$ in the even
stronger form that some measurable set of measure $2\eta \delta$
can be deleted from the interval $[-\delta, +\delta]$. In Section
\ref{sec:Bernstein} we prove that values of an idempotent,
concentrating on the grid, does not take too different values on
the intervals of length $2\delta$. Here we may consider additional
assumptions on the degree of the polynomials. Based on the results
of these sections, we will prove $p$-concentration for measurable
sets when $p>1/2$, with some estimates on constants. We conclude
the proof of Theorem \ref{th:concentration-meas} finally in \S
\ref{sec:measconclusion}.

\section{The use of Diophantine Approximation}\label{sec:Diophantine}

We will state two propositions, used respectively on $\GG$ and
$\GS$. The first one is a direct corollary of Khintchine's
Theorem, while the second one is its inhomogeneous extension,
first proved by Sz\"usz \cite{Sz} and later generalized by Schmidt
\cite{Sch}.

\begin{proposition} \label{grid-one}Let $E$ be a measurable set of positive measure in $\TT$. For all $\theta>0$,
$\eta>0$ and $Q\in \mathbb{N}$, there exists an irreducible
fraction $k/q$ such that $q>Q$ and
\begin{equation}\label{khin}
\left|\left [\frac kq-\frac{\theta}{q^2},\frac
kq+\frac{\theta}{q^2}\right]\cap E\right|\geq
(1-\eta)\frac{2\theta} {q^2}.
\end{equation}
 Moreover, given a positive integer $\nu$, it is possible to
choose $q$ such that $(\nu, q)=1$.
\end{proposition}

\begin{proposition} \label{grid-half}Let $E$ be a measurable set of positive
measure in $\TT$. For all $\theta>0$, $\eta>0$ and $Q\in
\mathbb{N}$, there exists an irreducible fraction $(2k+1)/(2q)$
such that $q>Q$ and
\begin{equation}\label{szusz}
\left|\left [\frac {2k+1}{2q}-\frac{\theta}{q^2},\frac {2k+1}{2q}
+\frac{\theta}{q^2}\right]\cap E\right|\geq (1-\eta)\frac{2\theta}
{q^2}.
\end{equation}
Moreover, given a positive integer $\nu$, it is possible to choose
$q$ such that $(\nu, q)=1$.
\end{proposition}
\begin{proof}[Proof of Propositions \ref{grid-one} and \ref{grid-half}]
Let $\alpha$ be $0$ or $1/2$. Then according to Sz\"usz' Theorem
\cite{Sz} for $\xi$ belonging to a set of full measure,
\begin{equation}\label{eq:qapproxhalf}
\|q\xi-\alpha\| \leq \frac{\theta}{q }
\end{equation}
has an infinite number of solutions. For $\alpha=1/2$, for
instance, it means that with a certain $k\in\NN$ ($0\leq k <q$) we
have
\begin{equation}\label{eq:qaproxhalfreduced}
|q\xi-1/2-k|<\frac{\theta}{q}, \qquad \textrm{i.e.} \qquad
\left|\xi-\frac{2k+1}{2q}\right|<\frac{\theta}{q^2}.
\end{equation}
We may  assume, and we will do it, that the denominator and
numerator are coprime: if not, we cancel out the common factors,
and the error, compared to the new denominator $q'$, is even
better. Note that for irrational $\xi$ we have infinitely many
different such denominators $q'$: indeed, if not we get a
contradiction with the fact that the error tends to zero with $q$.

Let us choose for $\xi$ an irrational density point of $E$ having
infinitely many solutions of \eqref{eq:qapproxhalf}. This we can
do, since almost every point of $E$ is such. For $\eta$ fixed and
$q$ sufficiently large we then have
$$\left|\TT\setminus E \cap\left[\xi-\frac{2\theta}
{q^2},\xi+\frac{2\theta }{q^2}\right]\right|\leq \frac{2\eta
\theta}{q^2}.$$ So, if $q$ and $k$ are such that
(\ref{eq:qaproxhalfreduced}) holds and if $q$ is large enough,
then (\ref{szusz}) is satisfied by the triangle inequality.

It remains to prove that the denominators $q$ can be taken so that
$(\nu, q)=1$. Schmidt proves in \cite{Sch} that, for each
polynomial $P$ with integer coefficients and each $\alpha\in \TT$,
for almost every $\xi$ one can find an infinite number of integers
$r$ such that
\begin{equation}\label{eq:qapproxnu}
\|P(r)\xi-\alpha\| \leq \frac{\theta}{r }.
\end{equation}
Both for $\alpha=0$ or $1/2$, it suffices to consider $P(r)=\nu r
+ 1$. Schmidt's Theorem then allows \eqref{eq:qapproxnu} for a.e.
$\xi$ by infinitely many $r$. So we can approach $\xi$ for
$\alpha=0$ by fractions $k/(\nu r+1)$, and for $\alpha=1/2$ by
fractions $\frac{2k+1}{2(\nu r+1)}$, eventually simplified. So the
denominator and $\nu$ will always remain coprime. The rest of the
proof is identical.
\end{proof}

\section{Peaking idempotents at $0$ and $1/2$}\label{s:peak-meas}
We will prove the following, which is a more accurate statement
than those of Section \ref{sec:onepeak}.
\begin{proposition}\label{p:peak-meas}
Let $p> 2$. For $\varepsilon>0$ there exists  $\delta_0>0$ and
$\eta>0$ such that, for all $\delta<\delta_0$ and $N\in
\mathbb{N}$,  if $E$ is a measurable set  that satisfies $|E\cap
[-\delta, \delta]|>2(1-\eta)\delta$, then there exists an
idempotent $T$ with gaps larger than $N$ such that
$$\int_{E\cap
[-\delta, \delta]} |T|^p>(1-\varepsilon)\int_0^1 |T|^p.$$ Let
$p>0$ not an even integer. Then for $\varepsilon>0$ there exists
$\delta_0>0$ and $\eta>0$ such that, for all $\delta<\delta_0$ and
$N\in \mathbb{N}$,  if $E$ is a measurable set  that satisfies
$|E\cap [\frac 12-\delta, \frac 12+\delta]|>2(1-\eta)\delta$, then
there exists an idempotent $T$ with gaps larger than $N$ such that

$$\int_{E\cap
[\frac 12-\delta, \frac 12+\delta]} |T|^p>(1-\varepsilon)\int_0^1
|T|^p.$$
\end{proposition}
\begin{proof}
We will proceed as in Section \ref{sec:onepeak}. The main point
is, for our peaking bivariate functions $f$, to find an
appropriate power $L$ of the marginal function $F$ for which the
same kind of estimate is valid: we will then take a Riesz product
with $L$ factors. The proposition will be a consequence of the
following lemma, with $F$ the associated marginal function.

\begin{lemma}\label{peak-measstrong} Let $F:[0,1/2]\rightarrow
[0,\infty)$ be a nonnegative, continuous function, having a strict
global maximum at $0$. Moreover, assume that there exist $0<a<A$
and $\Delta>0$ with $F$ admitting the estimates
\begin{equation}\label{eq:exptwosides}
F(0) \exp(-Ax^2) < F(x) < F(0) \exp(-ax^2) \qquad (x\in
[0,\Delta]).
\end{equation}

Then for all $\e>0$ there exists an $\eta>0$ so that for any
$0<\delta<\Delta$ there is an $L=L(\e,\de)\in \NN$ with the
property that whenever $E\subset [0,1/2)$ is a measurable set
satisfying $|E\cap[0,\delta]|>(1-\eta)\de$, then we have the
inequality
$$
\int_{E\cap [0, \delta]} F^L >(1-\varepsilon)\int_0^{1/2} F^L.
$$
\end{lemma}
\begin{remark} Observe that \eqref{eq:exptwosides} certainly holds
true in case $F$ has a nonvanishing second derivative (from the
right) at 0. Also note the validity of the obvious modification
for even functions on $[-1/2,1/2]$ assuming the analogous
two-sided conditions.
\end{remark}
\begin{proof} We can assume $F(0)=1$. By condition, $
\max_{[\Delta,1/2]} F  <1$, hence -- perhaps with a different
value of $a$, which still depends only on $F$ -- we have
$F(x)<\exp(-ax^2)$ on the whole of $[0,1/2]$. Extending $F$ to the
halfline $[0,\infty)$ as 0 outside $[0,1/2]$, we thus still have
this estimate.

Let now $H:=[0,\infty)\setminus ([0,\de]\cap E)$. Then we have
\begin{equation}\label{eq:Hintegral}
\int_H F^L < \int_{[0,\de]\setminus E} 1 + \int_\de^\infty
e^{-Lax^2}dx < |[0,\delta]\setminus E| + \int_\de^\infty
\frac{dx}{Lax^2} < \eta  \delta + \frac{1}{\de a L}.
\end{equation}
On the other hand with a very similar calculation we obtain
\begin{align}\label{eq:intonE}
\int_{[0,\de]\cap E} F^L & > \int_{[0,\de]\cap E} e^{-LAx^2} dx =
\left( \int_0^\infty - \int_{[0,\delta]\setminus E} -
\int_\delta^\infty \right) e^{-LAx^2} dx \notag \\ & \geq \frac12
\sqrt{\frac{\pi}{LA}}-|[0,\delta]\setminus E| - \int_\de^\infty
\frac{dx}{Lax^2}  > \frac12 \sqrt{\frac{\pi}{LA}}- \eta\delta -
\frac{1}{\de a L}.
\end{align}
A combination of \eqref{eq:Hintegral} and \eqref{eq:intonE}
reveals that it suffices to ascertain
$$
\eta \delta < \frac{\e}8 \sqrt{\frac{\pi}{LA}} \qquad \textrm{and}
\qquad \frac{1}{\de a L} < \frac{\e}8 \sqrt{\frac{\pi}{LA}},
$$
that is, with a constant $C=C(a,A)=C_F$,
$$
\eta \delta \sqrt{L} \leq \e' \qquad \textrm{and} \qquad
\frac{1}{\de \sqrt{L}} \leq \e' \qquad \textrm{with} \quad
\e':=\e/C.
$$
Thus we conclude the proof choosing $L:=\lceil \e'^{-2}\de^{-2}
\rceil$ and $\eta=\e'^2/2$.
\end{proof}

To prove both cases of the proposition, note that we can also
translate $F$ so that the maximum point falls to $1/2$ instead of
$0$.

At this point the proof of the proposition is identical to the
proofs of Section \ref{sec:onepeak}, using Lemma
\ref{l:Rieszconv}. For the given $E$, we find an idempotent $T$
such that integrals of $|T|^p$, respectively on $E\cap[-\delta,
+\delta]$ and on the whole torus, satisfy the same inequality as
the corresponding integrals for the function $F^L$.
\end{proof}

\section{Bernstein-type inequalities}\label{sec:Bernstein}

In order to adapt our proof of Proposition \ref{discret2cont}, we
need to control the error done when replacing values of
idempotents in a neighborhood of one of the grids by its values on
the grid.

We introduce the following notation, which will simplify the
proofs. For $f$ a periodic function, we will  use the sums of its
values on the two grids, which we denote by
\begin{equation}\label{def:sigma} \Sig_q(f):=
\sum_{k=0}^{q-1}f\left(\frac kq\right), \qquad \qquad
\Sig_q^\star(f):= \sum_{k=0}^{q-1}f\left(\frac{2k+1}{2q}\right).
\end{equation}

The aim of this paragraph is to recall classical inequalities, and
modify them according to our purposes. Let us prove the following
lemma.
\begin{lemma}\label{l:Bernstein}
For $1< p<\infty$ there exists a constant $C_p$ such that, for
$P\in \mathcal{T}_q$  and for $|t|<1/2$, we have the two
inequalities
\begin{equation}\label{maj-grid}
\sum_{k=0}^{q-1}|P( t+k/q)|^p\leq C_p \sum_{k=0}^{q-1}|P(k/q)|^p,
\end{equation}
\begin{equation}\label{maj-grid1}
\sum_{k=0}^{q-1}\left||P(t+k/q)|^p-|P(k/q)|^p \right| \leq C_p
|qt| \sum_{k=0}^{q-1}|P(k/q )|^p.
\end{equation}
\end{lemma}
\begin{proof}
For $1<p<\infty$, the $L^p$ norm of a trigonometric polynomial in
$\mathcal{T}_q$ is equivalent to the $\ell^p$ norm of its values
on the grid $\mathbb{G}_q$. This is known as the
Marcinkiewicz-Zygmund Theorem: the implied constants depend only
on $p$ but tend to $\infty$ for $p$ tending to $1$ or $\infty$.
For the exact form fitting to our Taylor polynomials see Theorem
(7.10), p. 30 chapter X in \cite{Z}; see also \cite{OS} for recent
extensions. Inequality (\ref{maj-grid}) then follows using the
Marcinkiewicz-Zygmund Theorem twice, and invariance by translation
of the $L^p$ norm.

To obtain (\ref{maj-grid1}), we use a variant of Bernstein's
Inequality, which may be stated, for $P\in\mathcal{T}_q$, as
\begin{equation}\label{Bernst}
 \int_0^1|P(x+t)-P(x)|^pdx\leq (2\pi q|t|)^p\int_0^1 |P(x)|^pdx.
\end{equation}
Since this is not the usual form of Bernstein's Inequality, we
indicate how to obtain it. We write, for positive $t$,
$$
|P(x+t)-P(x)|^p\leq t^{p-1}\int_x^{x+t}|P'(u)|^p du,
$$
apply this estimate on the left hand side of \eqref{Bernst} and
then change the order of integration. We then conclude by using
Bernstein's Inequality as stated in Theorem (3.16), chapter X in
\cite{Z}, that is,
$$ \int_0^1|P'(x)|^p dx\leq (2\pi q)^p\int_0^1 |P(x)|^pdx.$$
Let us proceed with the proof of (\ref{maj-grid1}). By using the
Marcinkiewicz-Zygmund Theorem for both sides of \eqref{Bernst}, we
find that, for $1<p<\infty$, there exists some constant $C_p$,
(independent of $P\in\mathcal{T}_q$), such that
\begin{equation}\label{maj-grid0}
\sum_{k=0}^{q-1}|P(t+k/q)-P(k/q)|^p\leq C_p |qt|^p
\sum_{k=0}^{q-1}|P(k/q)|^p.
\end{equation}
Let us use the elementary inequality
\begin{equation}\label{elem}
\left| |a|^p-|b|^p \right| \leq p |a-b|
\left(|a|^{p-1}+|b|^{p-1}\right)
\end{equation}
and the H\"older Inequality together with \eqref{maj-grid0}, as
well as our notation given in \eqref{def:sigma}. We obtain the
estimate
\begin{align*}\sum_{k=0}^{q-1}&\left||P(t+k/q)|^p-|P(k/q)|^p \right|
\leq \\
& \qquad pC_p^{\frac 1p} |qt|\left(\Sig_q(|P|^p)\right)^{\frac1p}
\cdot \left(\Sig_q\left(\left(|P|^{p-1}+|P(t+\cdot)|^{p-1}
\right)^\frac{p}{p-1}\right)\right)^{\frac{p-1}{p}}.
\end{align*}
After having used Minkowski's inequality and the estimate
(\ref{maj-grid}), i.e. $\Sig_q(|P(t+\cdot)|^p)\leq C_p
\Sig_q(|P|^p)$, the last factor on the right hand side becomes
$C_p'\left(\Sig_q(|P|^p)\right)^{\frac{p-1}{p}}$, which concludes
the proof of \eqref{maj-grid1}.
\end{proof}

The following is an easy consequence of Lemma \ref{l:Bernstein}.
\begin{lemma}\label{l:Bernstein-half}For $1< p<\infty$ and with
the same constant $C_p$ as in Lemma \ref{l:Bernstein} we have the
following property. Whenever $P\in \mathcal{T}_{2q}$ satisfies
\begin{equation}\label{complement}
\sum_{k=0}^{q-1}|P(k/q)|^p \leq K \Sig_q^\star(|P|^p),
\end{equation}
then, for any $|t|<1/2$, we have the two inequalities
\begin{equation}\label{maj-gridhalf}
\sum_{k=0}^{2q-1}\left|P\left(\frac k{2q}+t\right)\right|^p\leq
C_p(K+1) \Sig_q^\star(|P|^p),
\end{equation}
\begin{equation}\label{maj-grid1half}
\sum_{k=0}^{2q-1}\left|\left|P\left(\frac
{k}{2q}+t\right)\right|^p-\left|P\left(\frac
{k}{2q}\right)\right|^p \right| \leq 2C_p (K+1) |qt|
\Sig_q^\star(|P|^p).
\end{equation}
\end{lemma}
This lemma explains why we introduce the next definition.
\begin{definition}\label{Condition} Let $0<p<\infty$ and $q\in\NN$.
We say that a polynomial $f$ satisfies the \emph{grid-condition
with  constant $K$}, if we have
\begin{equation}\label{Kcondi}
\sum_{k=0}^{q-1}\left|f\left(\frac{k}{q}\right)\right|^p\leq K
\sum_{k=0}^{q-1}\left|f\left(\frac{2k+1}{2q}\right)\right|^p,
\end{equation}
that is, with the notation \eqref{def:sigma}, $\Sig_q (|f|^p)\leq
K \Sig_q^\star(|f|^p)$.
\end{definition}
\begin{remark}\label{exact-q} When $P\in {\mathcal T}_q$,
\eqref{complement} -- i.e., the grid condition \eqref{Kcondi} for
$P$ -- holds with $K=C_p$ depending only on $p>1$: just use
\eqref{maj-grid} for the translated by $1/2q$ polynomial.
\end{remark}

We will use these considerations for products of such polynomials
as well.
\begin{lemma}\label{error} For $1/2< p<\infty$ there exists a constant
$A_p$ such that, whenever $Q\in {\mathcal T}_{2q}$ satisfies the
grid-condition \eqref{Kcondi} with exponent $2p$, i.e.
\begin{equation}\label{complementQ}
\Sig_q (|Q|^{2p}):=\sum_{k=0}^{q-1}|Q(k/q)|^{2p}\leq K
\Sig_q^\star(|Q|^{2p}),
\end{equation}
then for the product polynomial $R(x):=Q(x)Q((2q+1)x)$ we have for
all $|t|<1/2$ and for all $a:=0, 1, \dots, q-1$ the two
inequalities
\begin{equation}\label{un}
\sum_{k=0}^{q-1}\left|R\left(t+\frac {2k+1}{2q}\right)
\right|^{p}\leq (1+A_p (K+1)q^2|t|)\Sig_q^\star(|Q|^{2p}),
\end{equation}
\begin{equation}\label{deux} \left|R\left(t+\frac
{2a+1}{2q}\right) \right|^{p}\geq \left|R\left(\frac
{2a+1}{2q}\right)
\right|^{p}-A_p(K+1)q^2|t|\Sig_q^\star(|Q|^{2p}).
\end{equation}
\end{lemma}
\begin{proof}
Let us put, for $k=0, 1, \dots 2q-1$,
\begin{equation}\label{XY}
X_k(t):= \left|Q\left(t+\frac {k}{2q}\right) \right|^{2p}.
\end{equation}
Note that the two factors of $R$ take the same values on the grid
$\GS$. Moreover, since $Q\in{\mathcal T}_{2q}$ and $2p>1$, it
follows from Lemma \ref{l:Bernstein-half}, formula
\eqref{maj-grid1half} that
$$
\sum_{k=0}^{2q-1}|X_{k}(t)-X_{k}(0)|\leq
2C_{2p}(K+1)q|t|\Sig_q^\star(|Q|^{2p}).
$$
Let us pass to
$$
Y_k(t):= \left|R\left(t+\frac {k}{2q}\right)
\right|^{p}=\sqrt{X_k(t)X_k((2q+1)t)}.
$$
Using the Cauchy-Schwarz Inequality and the previous inequality,
we find for all $|t|\leq 1/2$
\begin{align*}\label{sigma_ineq}
\sum_{k=0}^{q-1}Y_{2k+1}(t)&\leq
\left(\sum_{k=0}^{q-1}X_{2k+1}(t)\right)^{\frac
12}\left(\sum_{k=0}^{q-1}X_{2k+1}((2q+1)t)\right)^{\frac 12}\\
&\leq
\left(\Sig_q^\star(|Q|^{2p})+\sum_{k=0}^{2q-1}|X_{k}(t)-X_{k}(0)|\right)^{\frac
12}\\
&\qquad\qquad\times\left(\Sig_q^\star(|Q|^{2p})+
\sum_{k=0}^{2q-1}|X_{k}((2q+1)t)-X_{k}(0)|\right)^{\frac
12}\\
&\leq (1+2C_{2p}(2q+1)q(K+1)|t| )\Sig_q^\star(|Q|^{2p}).
\end{align*}
We have proved (\ref{un}). We can write in the same way that
\begin{align*} Y_{2a+1}(t)^2&\geq
\left(X_{2a+1}(0)-\sum_{k=0}^{2q-1}|X_{k}(t)-X_{k}(0)|\right)\\
&\qquad\qquad \times
\left(X_{2a+1}(0)-\sum_{k=0}^{2q-1}|X_{k}((2q+1)t)-X_{k}(0)|\right),
\end{align*}
so that
$$
Y_{2a+1}(t)\geq
Y_{2a+1}(0)-2C_{2p}(K+1)(2q+1)q|t|\Sig_q^\star(|Q|^{2p}).
$$
\end{proof}

\section{From discrete concentration to concentration for
measurable sets}\label{sec:measconclusion}

\begin{definition}\label{def:gammap-sharp} We define
\begin{equation}\label{gam-sharp}
\gamma_p^\sharp:= \liminf_{q\rightarrow \infty}
\gamma_p^\sharp(q), \qquad \gamma_p^\sharp(q):=\sup_{R\in \PP_q}
\frac{\left|R\left(\frac 1q\right) \right|^p}{ \sum_{k=0}^{q-1}
\left|R\left(\frac kq\right) \right|^p}.
\end{equation}
\end{definition}
Using the notation \eqref{eq:Bdef}, the results of Section
\ref{sec:unilower} give immediately
\begin{equation}\label{eq:gammasharpest}
(\gamma_p^\sharp)^{-1}\leq \inf_{0<t<1/2} B(p,t),
\end{equation}
valid for any $p>1$.

Let us give the corresponding definition for the grid $\GS$.
\begin{definition}\label{def:gammap-star} We define
\begin{equation}\label{gam-star}
\gamma_p^\star:= \liminf_{q\rightarrow \infty} \gamma_p^\star(q),
\qquad \gamma_p^\star(q):=\sup_{R\in \PP_q}
\frac{\left|R\left(\frac 1{2q}\right) \right|^p}{ \sum_{k=0}^{q-1}
\left|R\left(\frac {2k+1}{2q}\right) \right|^p}.
\end{equation}
\end{definition}
We have seen in Section \ref{sec:pcondiconc} that, for $p>1$, with
the notation \eqref{eq:Alimit},
\begin{equation}\label{eq:gammastarest}
(\gamma_p^\star)^{-1}\leq \inf_{0<t<1/2} A(p,t).
\end{equation}

\comment{We have the following proposition, which will be improved
later on by the use of probability methods. We restrict first to
$p$ not an even integer and the use of the grid $\GS$.}

\begin{proposition}\label{2pto-p}
Let $p>1/2$ not an even integer. Then there is $p$-concentration
for measurable sets, and $\gamma_p \geq 2\gamma_{2p}^{\star}$;
furthermore, there is $p$ concentration for measurable sets with
gap at the same level.
\end{proposition}

\begin{proof}

The proof is organized as the one of Proposition
\ref{discret2cont}. At the outset we have a measurable and
symmetric set $E\subset \TT$ with $|E|>0$. Let us first take
$C<\gamma_{2p}^{\star}$ arbitrarily close to
$\gamma_{2p}^{\star}$, then fix $\varepsilon$ a small constant.
Let $\eta$ and $\delta_0$  be given by Proposition
\ref{p:peak-meas} (second case), depending on $\e$. Let $\theta>0$
be a small constant which will be fixed later on, and $q_0$ large
enough so that, for $q>q_0$ one has $C<\gamma_{2p}^{\star}(q)$ and
$\theta/q\leq \delta_0$. With this data we consider some interval
centered at $(2a+1)/(2q)$ given by Proposition \ref{grid-half}.
Let $P\in\PP_{q}$ be such that
\begin{equation}\label{eq:Cdef}
\left|P\left(\frac 1{2q}\right) \right|^{2p} >~~ C ~~{
\sum_{k=0}^{q-1} \left|P\left(\frac {2k+1}{2q}\right)
\right|^{2p}} = C \SQS(|P|^{2p}).
\end{equation} By Lemma \ref{l:homominvqq}, and Remark \ref{degree-half},
we can find an idempotent $Q\in\PP_{2q}$ such that we have
\begin{equation}\label{R-at-a}
\left|Q\left(\frac{2a+1}{2q}\right)\right|^{2p}=
\left|P\left(\frac{1}{2q}\right)\right|^{2p}
\end{equation}
and
\begin{equation}\label{R-on-grid}
\sum_{k=0}^{q-1} \left|Q\left(\frac {2k+1}{2q}\right)
\right|^{2p}=\sum_{k=0}^{q-1} \left|P\left(\frac {2k+1}{2q}\right)
\right|^{2p},
\end{equation}
and also
\begin{equation}\label{R-on-grid0}
\sum_{k=0}^{q-1} \left|Q\left(\frac {k}{q}\right)
\right|^{2p}=\sum_{k=0}^{q-1} \left|P\left(\frac {k}{q}\right)
\right|^{2p}.
\end{equation}
Recall that $P\in \PP_q$, so for $2p>1$ according to Remark
\ref{exact-q} it satisfies the grid condition \eqref{Kcondi} with
a constant $C_{2p}$ depending only on $p$. Since $P$ and $Q$
attain exactly the same set of values both on the two grids $\GG$
and $\GS$,  the idempotent $Q$ also satisfies the grid-condition
\eqref{Kcondi} for $2p$ with the constant $C_{2p}$. So the
idempotent
\begin{equation}\label{eq:RdefprodP}
R(x):=Q( x) Q((2q+1) x),
\end{equation}
matching with $Q^2$ on both grids, also satisfies
\begin{equation}\label{eq:RQ2gridcondi}
|R(0)|^{p}\leq \SQ(|R|^{p}) \leq C_{2p}\SQS(|R|^{p}),
\end{equation}
i.e. the grid condition \eqref{Kcondi} holds for $R$, too (with
$K=C_{2p}$). Whence Lemma \ref{error} applies to $R$, so choosing
$\theta$ satisfying $A_p(C_{2p}+1)C^{-1}\theta \leq \e$ and in
view of \eqref{R-at-a}, \eqref{R-on-grid} and \eqref{R-on-grid0}
for all $|t|<\theta/q^2$ we obtain the estimates
\begin{equation}\label{une}
\sum_{k=0}^{q-1}\left|R\left(t+\frac {2k+1}{2q}\right)
\right|^{p}\leq (1+\e)\SQS(|P|^{2p})=(1+\e)\SQS(|R|^{p}),
\end{equation}
\begin{equation}\label{deuxe}
\left|R\left(t+\frac {2a+1}{2q}\right) \right|^{p}\geq
(1-\e)\left|R\left(\frac {2a+1}{2q}\right) \right|^{p},
\end{equation}
using also that, on comparing \eqref{eq:Cdef}, \eqref{R-at-a},
\eqref{R-on-grid} and \eqref{eq:RdefprodP} we are led to
\begin{equation}\label{eq:R2aplusone}
C\SQS(|R|^{p}) \leq \left|R\left(\frac {2a+1}{2q}\right)
\right|^{p}.
\end{equation}

Next, we will need a peaking idempotent at $1/2$, as obtained by
Proposition \ref{p:peak-meas}. This one will depend on our given
constants $\e$, $\eta$, $\delta=\theta/q$ and $N$ larger than the
degree of $R$, and also on a measurable set of finite measure
$E_\e$ that we define now. The mapping $x\mapsto qx$ is bijective
from $J:=(k/q,(k+1)/q )$ onto $(0,1)$, and we take for $E_\e$ the
image of $E\cap J$. It is clear that the condition
$$
|E_\e\cap [\frac 12-\delta, \frac 12+\delta]|>2(1-\eta)\delta
$$
has been satisfied. We take the idempotent $T$ provided by
Proposition \ref{p:peak-meas} for this data, satisfying
\begin{equation}\label{peak00}
\int_{E_\e\cap [\frac 12-\delta, \frac 12+\delta]}
|T|^p>(1-\varepsilon)\int_0^1 |T|^p.
\end{equation}
We finally consider the product
\begin{equation}\label{eq:Tdefmeas}
S(x):=T(qx) R(x),
\end{equation}
which is also an idempotent. We will prove as in Section
\ref{sec:oneconc} that
\begin{equation}\label{eq:SonTTm}
2C\int_{\TT} |S|^p\leq \kappa(\e)\int_E |S|^p ,
\end{equation}
with $\kappa(\varepsilon)$ being arbitrarily close to $1$ when
$\varepsilon$ is sufficiently small. In order to do this, we put
\begin{equation}\label{eq:JIdef}
J_k:=\left[\frac{k}{q},\frac{k+1}{q} \right], ~~ \qquad
I_k:=\left[ \frac{2k+1}{2q} -\frac{\theta}{q^2}, \frac{2k+1}{2q}+
\frac{\theta}{q^2} \right]
\end{equation}
for  $k=0,\dots,q-1$.
 From now on the proof of
the proposition is similar to the one of Proposition
\ref{discret2cont}. We repeat briefly the steps for the reader's
convenience. Denoting $\tau^p:=\int_{\TT}|T|^p$, we find, using
the property (\ref{peak00}), that
\begin{align}\label{eq:SonEmeas}\notag
\frac 12 \int_E |S|^p \geq \int_{I_a\cap E} |S|^p & \geq (1-
\varepsilon)\left|R\left(\frac{2a+1}{2q}\right)\right|^p \int_{I_a\cap E} |T(qx)|^pdx \\
& \geq (1-\varepsilon)\left|R\left(\frac{2a+1}{2q}\right)\right|^p
~~\frac{1}{q} \int_{E_\e\cap [\frac 12-\delta, \frac 12+\delta]} |T|^p \\
& \geq \frac{(1-\varepsilon)^2\tau^p}{q}
\left|R\left(\frac{2a+1}{2q}\right)\right|^p.\notag
\end{align}
Then we give an upper bound for the integral on the whole torus:
\begin{align*}
\sum_{k=0}^{q-1}\int_{I_k} |S|^p& =\int_{-\theta/q^2}^{\theta/q^2}
\sum_{k=0}^{q-1}\left|R\left(\frac{2k+1}{2q}+t\right)\right|^p
|T(qt)|^p ~dt \\ & \leq (1+\e)\SQS(|R|^{p})\frac{\tau^p}{q},
\end{align*}
while
\begin{align*}
\int_{J_k\setminus I_k}|S|^p&\leq 2\|R\|_\infty ^p \int_{\frac
kq+\frac \delta q}^{\frac kq+\frac 1{2 q}} |T(qx)|^p dx=2
\left|R(0)\right|^p \frac1q \int_{\frac \delta q}^{\frac 1{2}}
|T(x)|^p dx\\&\leq \frac {\varepsilon \tau^p}q
\left|R(0)\right|^p\leq \frac {C_{2p}\varepsilon \tau^p}q
\SQS(|R|^{p}),
\end{align*}
making use of \eqref{eq:RQ2gridcondi}, too. Summing the last
integrals over $k$, we obtain
\begin{equation}\label{eq:SonTmeas}
\int_{\TT} |S|^p \leq
\frac{\tau^p}{q}(1+\e+C_{2p}\e)\SQS(|R|^{p}).
\end{equation}
Now \eqref{eq:R2aplusone}, \eqref{eq:SonEmeas} and
\eqref{eq:SonTmeas} give \eqref{eq:SonTTm} with $\kappa(\e):=
$\hbox{$(1+\e+C_{2p}\e)(1-\varepsilon)^{-2}$}, concluding the
proof, except for assuring arbitrarily large gaps.
\medskip

It remains to indicate how to modify the proof to get peaking
idempotents with arbitrarily large gaps. So we fix $\nu$ as a
large odd integer, and we will prove that we can replace the
polynomial $Q(x)$ by some polynomial $\tilde Q(\nu x)$, with gaps
at least $\nu$. Recall first that we can take arbitrarily large
$q$ satisfying $(\nu,q)=1$. So we now choose $\tilde Q $ similarly
as before, to be the polynomial of degree $2q$ that coincides with
$P(bx)$ on the grid $\GG$, but now with $b$ chosen so that $\nu b
(2a+1)\equiv 1 $ mod $2q$. Such a $b$ exists, as $\nu(2a+1)$ and
$2q$ are coprime. We then fix
$$
R(x):=\tilde Q(\nu x) \tilde Q((2q+1)\nu x).
$$
There is an additional factor $\nu$, which modifies the value of
$\theta$, but otherwise the proof is identical. We know that
$\tilde Q(\nu x)$ and $P(bx)$, and thus $P(x)$, take globally the
same values on both grids $\GG$ and $\GS$, because in each case we
multiply by an odd integer that is coprime with $2q$. So in
particular the grid condition \eqref{complementQ} is satisfied
with $C_{2p}$ once again.
\end{proof}

Similarly, but with the grid $\GG$ instead of $\GS$, we obtain the
following.
\begin{proposition}\label{pto-p-int}
Let $p>2$ an even integer. Then there is $p$-concentra-tion  for
measurable sets, and $\gamma_p\geq 2 \max\left( \gamma_{p}^\sharp,
\gamma_{2p}^\sharp\right)$. Moreover, we can choose the
concentrating trigonometric polynomials with arbitrarily large
gaps.
\end{proposition}
\begin{proof}
We do not give the proof, since most modifications  are
straightforward, and even simpler. Now if $\gamma_{p}^\sharp\geq
\gamma_{2p}^\sharp$, we consider $C<\gamma_{p}^\sharp$ and $P$
satisfying
\begin{equation}\label{eq:Cdef+int}
\left|P\left(\frac 1{q}\right) \right|^{p} >~~ C ~~{
\sum_{k=0}^{q-1} \left|P\left(\frac {k}{q}\right) \right|^{p}}.
\end{equation}
We build $R:=Q:={\mathbf \Pi}_q P(b~\cdot)$ of degree lower than
$q$, using Lemma \ref{l:homominvariance} and Remark \ref{degree},
with $b$ chosen such that $b\cdot a\equiv 1 \mod q$, and thus
$a/q$ is mapped on $1/q$. Thus we obtain the required
concentration as above.

If $\gamma_{p}^\sharp < \gamma_{2p}^\sharp$, we take
$C<\gamma_{2p}^\sharp$ and an idempotent $P\in\PP_q$ satisfying
\begin{equation}\label{eq:Cdef2p}
\left|P\left(\frac 1{q}\right) \right|^{2p} >~~ C ~~{
\sum_{k=0}^{q-1} \left|P\left(\frac {k}{q}\right) \right|^{2p}}.
\end{equation}
In this case we consider $R:=R(x):=Q( x) Q((q+1) x)$ with
$Q:={\mathbf \Pi}_{q} P(b~\cdot)\in\PP_q$, and the proof is even
more like the above argument.

\end{proof}

\section{Positive definite trigonometric
polynomials}\label{sec:posdef}

The proof of Proposition \ref{2pto-p} generalizes directly to the
class $\PP^+$, with the main difference that, when considering the
values of a polynomial $P$ on some grid $\GG$ or $\GS$, we can
always consider the projected polynomial $\mathbf{\Pi}_{2q} (P)$,
taking the same values on ${\mathbb G}_{2q}$ and hence both on
$\GG$ and on $\GS$: here we need not be concerned for occasional
coincidences of projected terms in the sum, as the projection
$\mathbf{\Pi}_{2q}$ leaves $\PP^{+}$ invariant anyway. Therefore,
the concentration constants $\gamma^{+}_p$, that we will obtain
for positive definite functions and measurable sets, will be the
same as the ones for open sets (i.e. $c_p$). In particular, we
have the following.

\begin{theorem}\label{Lpto-p}
Let $p>0$ not an even integer. Then there is full
$p$-concentration for the class $\PP^+$ for measurable sets.
Moreover, we can choose the concentrating positive definite
trigonometric polynomials with arbitrarily large gaps.
\end{theorem}
\begin{proof} The proof follows the same lines as the one of Proposition
\ref{2pto-p}, but is simpler. We know that for $p\notin 2\NN$
there is full $p$-concentration at $1/2$, and also from Section
\ref{sec:pcondiconc} that this implies $c_p^{\star} = 1/2$, c.f.
the proof of Proposition \ref{prop:pcondiconc}. So it is
sufficient to prove the following lemma, which is very similar to
Proposition \ref{2pto-p}.
\end{proof}
\begin{lemma} Let $p>0$. Then there is $p$-concentration for the
class $\PP^+$ for measurable sets, and if $p\notin 2\NN$, then the
level of concentration satisfies $\gamma_p^{+}\geq 2
c^{\star}_{Lp}$ for any $L$ such that $Lp>1$. Moreover, unless
$p=2$, we can choose the concentrating trigonometric polynomials
with arbitrarily large gaps.
\end{lemma}
\begin{proof}
We only sketch the modifications to accomplish in the proof of
Proposition \ref{2pto-p}. Now $C<c_{Lp}^\star$. Naturally, we
choose $P\in\PP_{q}$ such that,
\begin{equation}\label{eq:Cdef+}
\left|P\left(\frac 1{2q}\right) \right|^{Lp} >~~ C ~~{
\sum_{k=0}^{q-1} \left|P\left(\frac {2k+1}{2q}\right)
\right|^{Lp}}.
\end{equation}
Then, as before, we choose $Q:=\mathbf{\Pi}_{2q}(P(b~\cdot))$. Now
we can take $R:=Q^L$, as clearly $R\in \PP^+$, and its degree is
less than $2Lq$ (instead of $2q(2q+1)$ previously). So the
Bernstein type inequalities can be applied more easily, with
better estimates than previously, not restricting the value of $L$
in this case. (In fact, we could as well consider
$\mathbf{\Pi}_{2q} R \in {\mathcal T}_{2q}\cap \PP^{+}$, too.)

Note that here there is no need to $L\to \infty$, but only to take
some $L>1/p$, as we already have $c_p^{\star}=1/2$ for $p\notin
2\NN$. On the other hand $L>1/p$ we really do need, as we apply
Marcinkievicz-Zygmund inequalities in the proof.

Otherwise the proof for $Lp>1$ can be adapted from Proposition
\ref{2pto-p}, with all other modifications being straightforward.
\end{proof}

When $p\in2\NN$, we do not have gap-peaking at $1/2$, but, unless
$p=2$, we have that at 0. With a completely analogous argument, we
obtain the corresponding result as follows.

\begin{theorem}\label{Lpto-pNN}
Let $p\ne 2$ be an even integer. Then there is $p$-concentra\-tion
for the class $\PP^+$ for measurable sets at the level
$\gamma_p^{+}\geq 2 \sup_{L\in \NN} c_{Lp}^\sharp$. Moreover, we
can choose the concentrating positive definite trigonometric
polynomials with arbitrarily large gaps.
\end{theorem}

\section{Concentration of random idempotents}\label{sec:random}

We will see that part of the estimates proved for $\PP^+$ in
Section \ref{sec:posdef} extend to $\PP$. This will be shown by
certain random constructions of idempotents.

We have seen in Section \ref{sec:unilower} that $\inf_t
B(\lambda,t)$ appears naturally when proving lower bounds for
$c_p$ when $p>2$ is an even integer: for $c_p$ (and thus for
$\gamma_p^+$) we obtained the lower bound $\sup_L 2/\inf_t
B(Lp,t)$. We will now prove the same lower bound for $\gamma_p$.
\begin{proposition}\label{Lpto-p-int}
Let $p>2$ an even integer. Then, for $L\geq 1$ an integer,
$\gamma_{p} \geq 2/\inf_t B(Lp,t)$.
\end{proposition}
\begin{proof}
Let $C<1/\min_t B(Lp,t)=1/B(Lp,t_0)$, say, and let us chose some
$c:=c(L,p)<t_0$. Then let $q$ be large enough, and $P\in\PP_q$
such that
\begin{equation}\label{eq:Cdef+int2}
\left|P\left(\frac 1{q}\right) \right|^{Lp} >~~ C ~~{
\sum_{k=0}^{q-1} \left|P\left(\frac {k}{q}\right) \right|^{Lp}}.
\end{equation}
Reflecting back to Section \ref{sec:unilower}, we know that $P$
may be taken as some Dirichlet kernel $D_r$, with $r=[t_0q]>cq$.
(This is the only specific property of $D_r$ that we will use.)
Let us take $R:=M^{-1}\mathbf{\Pi}_q(P^L)$, which coincides with
$M^{-1}P^L$ on the grid $\GG$. Choosing $M:=Lr^{L-1}$, which is a
majorant of the Fourier coefficients of $\mathbf{\Pi}_q(P^L)$, the
polynomial $R$ may be written as
$$R=\sum_{k=0}^{q-1}\alpha_k e_k,$$
with all $\alpha_k\in [0,1]$ and $\sum_k \alpha_k=R(0) =r/L$. By
construction, we also have
\begin{equation}\label{eq:Rconcenratestoo}
\left|R\left(\frac 1{q}\right) \right|^{p} >~~ C ~~{
\sum_{k=0}^{q-1} \left|R\left(\frac {k}{q}\right) \right|^{p}}.
\end{equation}

We now define a random idempotent $R_\omega$ by
$$R_\omega =\sum_{k=0}^{q-1}X_k(\omega) e_k,$$
where $X_k$ are independent Bernoulli random variables, with $X_k$
of parameter $\alpha_k$, that is, $\mathbb{P}(X_k=1)=\alpha_k$. We
want to prove that for any $\e>0$ and for $q>q_0(\e)$, with
positive probability the random idempotent $R_\omega$ satisfies
the inequality
\begin{equation}\label{eq:Cdef-rand}
\left|R_\omega\left(\frac 1{q}\right) \right|^{p} >~~K(\e) ~~{
\sum_{k=0}^{q-1} \left|R_\omega\left(\frac {k}{q}\right)
\right|^{p}},
\end{equation}
with $K(\e):=K_p(\e)$ arbitrarily close to $C$ with $\e$
sufficiently small.

Observe that our random idempotents $R_\omega$ are such that
${\mathbb E}(R_\omega(x)) =R(x)$, so in view of
\eqref{eq:Rconcenratestoo}, in order to prove \eqref{eq:Cdef-rand}
we have to measure the error done when replacing $R_\omega$ by its
expectation. Let us center our Bernoulli variables $X_k$ by
considering $\widetilde X_k:=X_k-\alpha_k$. Clearly, $\widetilde
X_k$ has variance ${\mathbb V}(\widetilde X_k)= \alpha_k
(1-\alpha_k)\leq \alpha_k$, so $R_\omega(k/q)$ has expectation
$R(k/q)$ and variance bounded by $r/L$. Also, by assumption,
$|R(1/q)|>C^{1/p}R(0)>\frac {cC^{1/p}}L q$, so after an
application of Markov's Inequality we find
$$
\mathbb{P}\left(\left|\frac{R_\omega(1/q)}{R(1/q)}\right| \leq
1-\e\right)\leq A \e^{-2}q^{-1},
$$ where $A$ depends on $p,c,L$, but is independent of $q$ and
$\e$. Whence for $q$ large enough, the inequality
\begin{equation}\label{minRrand}
\left|\frac{R_\omega(1/q)}{R(1/q)}\right|>1-\e
\end{equation}
holds with probability say at least $2/3$.

Let us now consider the sums
$$
S(\omega):=\sum_{k=0}^{q-1} \left|R_\omega\left(\frac
{k}{q}\right) \right|^{p}\qquad \qquad S:=\sum_{k=0}^{q-1}
\left|R\left(\frac {k}{q}\right) \right|^{p},$$ which we want to
compare. So we also put
$$
\widetilde R_\omega(k/q):=R_\omega(k/q)-R(k/q), \qquad
\widetilde{S}(\omega):= \sum_{k=0}^{q-1} \left|\widetilde
R_\omega(k/q) \right|^p.
$$
We claim that
\begin{equation}\label{martingale}
\mathbb{E}(\widetilde{S}(\omega)) \leq q C_p
\left(1+\sum\alpha_k\right)^{\frac p2} = q C_p\left(1+\frac
rL\right)^{\frac p2}.
\end{equation}

Let us first assume this inequality and conclude the proof of the
proposition. So, using \eqref{martingale},$S\geq R(0)^p=r/L$  and
$\widetilde{S}(\omega)\geq 0$ we are led to
$$
\mathbb{P}\left(C(\e) \widetilde{S}(\omega) \geq \e S \right)\leq
\frac{C(\e)}{\e S} \cdot q C_p \left(1+\frac rL\right)^{\frac
p2}\leq A\e^{-1}q^{1-p/2}.
$$
Therefore the inequality
\begin{equation}\label{majSrand}
C(\e)\widetilde{S}(\omega)
<\e S
\end{equation}
also holds with probability at least $2/3$ for $q$ large enough.

Next we will need the elementary inequality
\begin{equation}\label{eq:elementary}
|a|^p \leq (1+\e) |b|^p + C(\e) |a-b|^p,
\end{equation}
valid for arbitrary $\e>0$ with some corresponding constant
$C(\e)$. This is indeed obvious in case we have $|a|\leq \mu |b|$
with $\mu:=(1+\e)^{1/p}>1$, while otherwise we can write
$|a-b|\geq |a|-|b| \geq |a|( 1-1/\mu))$, therefore $|a|\leq
\mu/(\mu-1) |a-b|$ and we obtain the inequality again. So applying
this inequality with $a=R_\omega(k/q)$ and $b=R(k/q)$ we can
estimate $ |R_\omega(k/q)|^p$ by $
(1+\e)|R(k/q)|^p+C(\e)|\widetilde R_\omega(k/q)|^p$, yielding
$$
S(\omega) \leq (1+\e) S + C(\e)\widetilde{S}(\omega).
$$
Therefore, taking into account \eqref{majSrand},
\eqref{eq:Rconcenratestoo} and \eqref{minRrand}, we find that
\begin{align*}
CS(\omega)&\leq C(1+2\e)S<(1+2\e)|R(1/q)|^p \leq \frac
{1+2\e}{(1-\e)^p} |R_\omega(1/q)|^p
\end{align*}
holds with probability at least 1/3 for $q>q_0=q_0(\e,p,c,L)$.

So we find that \eqref{eq:Cdef-rand} does indeed hold with
$K(\e):=C(1-\e)^p/(1+2\e)$ and for some appropriate idempotent
$R_{\omega}$, once we have \eqref{martingale}, which we prove now.
This is a consequence of the following lemma, which is certainly
classical, but which we give here for the reader's convenience.
\begin{lemma}\label{mart-bern}For $p>1$ there exists some constant $C_p$
with the following property. Let $\alpha_k\in [0,1]$ and $a_k\in
\CC$ be arbitrary for $k=0,1,\dots,N$. Let $X_k$ be a sequence of
independent Bernoulli random variables with parameter $\alpha_k$,
and let $\widetilde X_k:=X_k-\alpha_k$ be their centered version,
again for $k=0,1,\dots,N$. Then we have
$$
\mathbb{E}\left(\left|\sum_{k=0}^Na_k \widetilde
X_k\right|^{2p}\right)\leq C_p\cdot \max_{k=1,\dots,N} |a_k|^{2p}
\cdot (1+\sum_{k=0}^N \alpha_k)^{p}.
$$
\end{lemma}
\begin{proof} We can normalize by taking $\max_{k=1,\dots,N} |a_k|=1$. It follows from classical martingale inequalities (see
\cite{Bur}) that
$$\mathbb{E}\left(\left|\sum_{k=0}^Na_k \widetilde
X_k\right|^{2p}\right)\leq A_p\mathbb{E}\left(\left|\sum_{k=0}^N
\widetilde X_k^2\right|^{p}\right).
$$
So we are left with proving the inequality
\begin{equation}\label{bernoulli}
\mathbb{E}\left(\left|\sum_{k=0}^N \widetilde
X_k^2\right|^{p}\right)\leq A'_p
\left(1+\sum_{k=0}^N\alpha_k\right)^{p}.
\end{equation}

If $0\leq \alpha\leq 1$ and $Y$ is a centered Bernoulli variable
with parameter $\alpha$, then
$$
{\mathbb E} \left(e^{Y^2}\right) =\al e^{(1-\al)^2} + (1-\al)
e^{\al^2} \leq \al (1+e(1-\al)^2)+ (1-\al)(1+e \al^2)
\leq e^{e\al},
$$
because $e^x\leq 1+ex$ for $0\leq x \leq 1$ and $1+
e\al(1-\al)\leq 1+e\al \leq e^{e\al}$. So
\begin{equation}\label{bernoulli2}
\mathbb{E}\left(e^{\sum_{k=0}^N \widetilde X_k^2}\right)\leq e^{e
\sum_{k=0}^N \alpha_k}.
\end{equation}
Finally, we use the fact that, whenever $Z$ is a nonnegative
random variable such that $\mathbb{E}(e^Z)\leq e^\kappa$, then
\begin{align*}\mathbb{E}(Z^p)&=p\int_0^\infty
{\mathbb P}(Z>\lambda)\lambda^{p-1}d\lambda
\leq (2\kappa)^p+p \int_{2\kappa}^\infty
e^{\kappa-\lambda}\lambda^{p-1}d\lambda \\& \leq 2^p \kappa^p+p
\int_0^\infty e^{-\lambda/2}\lambda^{p-1}d\lambda =
2^p\kappa^p+A''_p\leq (2^p+A''_p) (1+\kappa)^p.
\end{align*}
Putting $Z:=\sum_k \widetilde{X}_k^2$ and $\kappa:=e \sum_k\al_k$,
\eqref{bernoulli2} leads to \eqref{bernoulli}.
\end{proof}
So there exists $R_\omega\in \PP_q$ with \eqref{eq:Cdef-rand},
whence $\liminf_{q\to\infty} \gamma_p^\sharp(q)\geq C$, even
$\gamma^\sharp:=\liminf_{q\to\infty} \gamma_p^\sharp(q) \geq
1/\inf_t B(Lp,t)$, and referring to Proposition \ref{pto-p-int}
concludes the proof of Proposition \ref{Lpto-p-int}.
\end{proof}

Note that the result implies $\gamma_4\geq 2/ \inf_t
B(4,t)=0.495\dots$, as computed in \eqref{p=4} at the end of
Section \ref{sec:unilower} for the sake of $c_4$, and similarly
$\gamma_{2k}\geq 0.483\dots$ for general $k>2$ according to the
calculations of \eqref{majB}.

\begin{remark} These results could also have been obtained by
applying the direct estimates of Salem and Zygmund \cite{SZ},
which allow here to have estimates of the maximum value of
$|\widetilde{R}_\omega|$ on the grid $\mathbb{G}_q$. The same
remark holds for the next case, using the grid
$\mathbb{G}_{2Lq}$.\comment{With the above notations these give on
the $N=q$ equally spaced points $x_j:=j/q$ that
$$
\max_{j=1,\dots N} |\widetilde{R}_\omega(x_j)| \ll \sqrt{N}
\sqrt{{\mathbb V}\left(\sum_k X_k\right)} \ll {\sqrt{N \sum_k
\al_k }}\ll \sqrt{q \log q},
$$
with positive probability. The consequent estimates are then
entailed to the effect that $S(\omega)\ll q^{1+p/2+\e}$ for $p>2$.
So in view of $R(1/q)>cq$, clearly we obtain
$|\widetilde{R}_\omega(1/q)|=o(q^{1/2+\e})=o(|R(1/q)|)$, whence
$S(\omega)=o(q^p)=o(|R(1/q)|^p)$ also holds for $p>2$. The same
works even with $N=2Lq$ points, so in the next proof it can also
replace our calculations.}
\end{remark}

The use of the same methods for $p>2$ not an even integer is
somewhat more delicate: nevertheless, we will prove full
$p$-concentration with gap for measurable sets. According to
Proposition \ref{2pto-p}, it would suffice to show
$\gamma_p^{\star}=1/2$ for $p>2$. Essentially, we will do this,
but with some necessary modifications. On the other hand we do
know $c_p^{\star}=1/2$ e.g. from the proof of Proposition
\ref{prop:pcondiconc}: this proof also provides us a concrete
construction, with the product of certain Dirichlet kernels in the
proof, which we will make use in some extent. We start with
\begin{lemma}\label{full-meas}
Let $p>2$. Then for all $C<1/2$, there exists a constant
$K:=K_p(C)$ with the property that for $q$ large there exists an
idempotent $P\in\PP_{2q}$ which satisfies the two inequalities
\begin{align}\label{eq:grid-half-rand}
\left|P\left(\frac 1{2q}\right) \right|^{p} &>~~ C ~~{
\sum_{k=0}^{q-1} \left|P\left(\frac {2k+1}{2q}\right)
\right|^{p}},\\
\left|P\left(\frac 1{2q}\right) \right|^{p} &>~~ K ~~{
\sum_{k=0}^{q-1} \left|P\left(\frac {k}{q}\right) \right|^{p}}
.\label{replacement}
\end{align}
\end{lemma}
\begin{proof}
We use now from Section \ref{sec:pcondiconc} that for $L$ large
enough and $q$ large enough there exists an idempotent in $\PP_q$,
which actually can be taken some Dirichlet kernel $D_r$, with say
$r:=[q/4]>cq$ (for some fixed value of $c=c(L,p)<1/4)$, such that
$$
\left|D_r\left(\frac 1{2q}\right) \right|^{Lp} >~~ C ~~{
\sum_{k=0}^{q-1} \left|D_r\left(\frac {2k+1}{2q}\right)
\right|^{Lp}}.
$$
From now on we fix $L$, so that constants may as well depend on
$L$.

Next, we wish to ensure, with some constant $K=K(C,p,L)$, that
\begin{equation}\label{eq:quazigrid}
\left|D_r\left(\frac 1{2q}\right) \right|^{Lp} >~~ K ~~{
\sum_{k=0}^{q-1} \left|D_r\left(\frac {k}{q}\right) \right|^{Lp}}
.
\end{equation}
In view of the concrete form of the Dirichlet kernel, it is
obvious, that $|D_r(1/2q)|\geq |D_r(1/q)|$. Consider now,
recalling the estimation of the concentration constants
$c^\sharp_p(q)\to c^\sharp_p$ in Section \ref{sec:unilower}, and
in particular reflecting back to \eqref{prod-idem} --
\eqref{eq:Batzero}, the lower estimates
$$
\left|D_r\left(\frac 1q\right) \right|^{Lp} >~~
\frac{1}{B(Lp,[q/4],q)} ~~{ \sum_{k=0}^{q-1} \left|D_r\left(\frac
{k}{q}\right) \right|^{Lp}}.
$$
As $B(Lp,[q/4],q) \to B(Lp,1/4)>0$ ($q\to\infty$), this clearly
implies \eqref{eq:quazigrid}.

At this point we proceed as above. First we consider the
$L^{\rm{th}}$ power of $D_r$ and take for $P$ the projected
polynomial $M^{-1}\mathbf{\Pi}_{2q} (D_r^L)$, with $M:=Lr^{L-1}$ a
majorant of the Fourier coefficients of $D_r^L$. The polynomial
$P$ may be written as
$$P=\sum_{k=0}^{2q-1}\alpha_k e_k,$$
with all $\alpha_k\in [0,1]$ and $\sum \alpha_k=P(0)=r/L ~~~~(
\simeq c(L) q)$. So we have
\begin{equation}\label{eq:PCcondition}
\left|P\left(\frac 1{2q}\right) \right|^{p} >~~ C ~~{
\sum_{k=0}^{q-1} \left|P\left(\frac {2k+1}{2q}\right)
\right|^{p}}.
\end{equation}
Moreover, by construction we also have the grid condition
\begin{equation}\label{eq:Pgridtoo}
\left|P\left(\frac 1{2q}\right) \right|^{p} >~~ K ~~{
\sum_{k=0}^{q-1} \left|P\left(\frac {k}{q}\right) \right|^{p}}
\end{equation}
with a certain constant $K=K(C,p,L)$.

Observe that the only required property what $P$ does not have is
being an idempotent: here $P\in{\mathcal T}_{2q} \cap \PP^{+}$,
while we need some polynomial in $\PP_{2q}$. So we define, as
before, a random idempotent $P_\omega$ by
$$
P_\omega :=\sum_{k=0}^{2q-1}X_k(\omega) e_k,
$$
where $X_k$ are independent Bernoulli random variables, with $X_k$
of parameter $\alpha_k$, that is, $\mathbb{P}(X_k=1)=\al_k$. Then
again $P(x)=\mathbb{E}P_\omega(x)$, and we measure the error done
when replacing $P_\omega$ by its expectation.

Let us write $X_k=\alpha_k+\widetilde X_k$, where $\widetilde X_k$
is centered and has variance $\alpha_k (1-\alpha_k)\leq \alpha_k$.
So $P_\omega(k/(2q))$ has  expectation $P(k/(2q))$ and  variance
bounded by $r/L$.

By construction $|P(1/(2q))|>K^{1/p}P(0)>\frac {cK^{1/p}}L q$. So,
by Markov Inequality, as before, we find that for $q$ large
enough, the inequalities
\begin{equation}\label{minPrand2}
\left|\frac{P_\omega(1/(2q))}{P(1/(2q))}\right|>1-\e, \qquad
\left|\frac{P_\omega(1/q)}{P(1/q)}\right|>1-\e
\end{equation}
hold with probability $2/3$.

Denoting again $\widetilde P_\omega(x):=P_\omega(x)-P(x)$, let us
now consider the sums
\begin{align*}
&S(\omega):=\sum_{k=0}^{q-1} \left|P_\omega\left(\frac
{2k+1}{2q}\right) \right|^{p}, \qquad S'(\omega):=\sum_{k=0}^{q-1}
\left|P_\omega\left(\frac {k}{q}\right) \right|^{p}, \\ &
S:=\sum_{k=0}^{q-1} \left|P\left(\frac {2k+1}{2q}\right)
\right|^{p}, \qquad \qquad S':=\sum_{k=0}^{q-1} \left|P\left(\frac
{k}{q}\right) \right|^{p},\\& \widetilde{S}(\omega):=
\sum_{k=0}^{q-1} \left|\widetilde P_\omega \left(\frac
{2k+1}{2q}\right)\right|^p,\qquad
\widetilde{S}'(\omega):=\sum_{k=0}^{q-1}
\left|\widetilde{P}_\omega\left(\frac {k}{q}\right) \right|^{p} .
\end{align*}
To compare these again we use the elementary inequality
\eqref{eq:elementary} to get $|P_\omega(k/(2q))|^p\leq
(1+\e)|P(k/(2q))|^p+C(\e)|\widetilde P_\omega(k/(2q))|^p$ and thus
$$
S(\omega) \leq (1+\e) S + C(\e)\widetilde{S}(\omega), \qquad
S'(\omega) \leq (1+\e) S' + C(\e)\widetilde{S}'(\omega).
$$
Applying Lemma \ref{mart-bern} as before, analogously to
\eqref{martingale} we now obtain
$$
\mathbb{E}|\widetilde P_\omega(k/(2q))|^p\leq q C_p
\left(1+\sum\alpha_k\right)^{\frac p2}\leq c'(p,L) q^{1-p/2}.
$$
So for $q$ large enough, similarly to \eqref{majSrand}, we prove
as before that the inequalities
$$ C(\e) \widetilde S (\omega) <\e S,  \qquad C(\e)
\widetilde S' (\omega) <\e S'
$$
hold with probability $2/3$, thus combining with the above, we
even have
$$
S(\omega) < (1+2 \e) S , \qquad S'(\omega) < (1+2\e) S'
$$
with probability at least 1/3. Taking into account also
\eqref{eq:PCcondition}, \eqref{eq:Pgridtoo} and \eqref{minPrand2},
we can summarize our estimates so that with positive probability
\begin{align*}
CS(\omega) & < \frac {1+2\e}{(1-\e)^p}
\left|P_\omega\left(\frac {1}{2q}\right)\right|^p,\\
KS'(\omega) & < \frac {1+2\e}{(1-\e)^p} \left|P_\omega\left(\frac
{1}{2q}\right)\right|^p.
\end{align*}
Since $\e$ is arbitrary, we conclude that some $P_\omega\in
\PP_{2q}$ satisfies the requirements of the Lemma.
\end{proof}

At this point, we have all the elements to have the best constant
for all $p>1$ not even.

\begin{proposition}\label{full-rang}
Let $p>1$ not an even integer. Then there is full
$p$-concentration with gap for measurable sets.
\end{proposition}

\begin{proof} The proof follows the same lines as the proof of
Proposition \ref{2pto-p}. We take now $C<1/2$ and, instead of
choosing $P\in\PP_q$ satisfying \eqref{eq:Cdef} and starting the
consruction of $Q$ with that, we start with choosing
$P\in\PP_{2q}$ given by Lemma \ref{full-meas}, with exponent
$2p>2$.

Note that the only point of the proof of Proposition \ref{2pto-p}
using the fact that $P$ is in $\PP_q$ is the grid condition
\eqref{complementQ}, which is given now by \eqref{replacement}.
Thus Lemma \ref{error} applies even in this case, while otherwise
the proof is exactly as for Proposition \ref{2pto-p}.
\end{proof}

\end{document}